\documentclass[11pt]{amsart}%
\usepackage{amscd,amsmath,latexsym,amsthm,amsfonts,amssymb,graphicx,color,geometry}
\usepackage{enumerate}
\usepackage[pdftex,plainpages=false,pdfpagelabels,
colorlinks=true,citecolor=blue,linkcolor=blue,urlcolor=black,
filecolor=black,bookmarksopen=true,unicode=true]{hyperref}
\usepackage[norefs,nocites]{refcheck}
\usepackage{amsmath}
\usepackage{amsfonts}
\usepackage{amssymb}
\usepackage[T1]{fontenc}
\usepackage{graphicx}
\usepackage{tabularx}
\usepackage{xcolor}
\usepackage{colortbl}
\usepackage{float}
\usepackage{pgf,tikz,pgfplots}
\usepackage{tkz-base}
\usepackage{tkz-fct}
\usepackage{multicol}%
\providecommand{\U}[1]{\protect\rule{.1in}{.1in}}
\tikzset{>=Triangle}
\newtheorem{theorem}{Theorem}[section]
\newtheorem{proposition}[theorem]{Proposition}

\newtheorem{corollary}[theorem]{Corollary}

\newtheorem{remark}[theorem]{Remark}

\newtheorem{lemma}[theorem]{Lemma}

\numberwithin{equation}{section}
\pgfplotsset{compat=1.17}
\begin{document}
\title[Constants of Bennett's inequality and the Gale--Berlekamp switching game]{Upper bounds for the constants of Bennett's inequality and the Gale--Berlekamp
switching game}
\author[D. Pellegrino]{Daniel Pellegrino}
\address{Departamento de Matem\'{a}tica \\
Universidade Federal da Para\'{\i}ba \\
58.051-900 - Jo\~{a}o Pessoa, Brazil.}
\email{dmpellegrino@gmail.com and daniel.pellegrino@academico.ufpb.br}
\author[A. Raposo Jr.]{Anselmo Raposo Jr.}
\address{Departamento de Matem\'{a}tica \\
Universidade Federal do Maranh\~{a}o \\
65085-580 - S\~{a}o Lu\'{\i}s, Brazil.}
\email{anselmo.junior@ufma.br}
\thanks{D. Pellegrino is supported by CNPq Grant 307327/2017-5  }
\subjclass[2010]{05A05, 91A46, 15A60, 15A51, 47A07}
\keywords{Combinatorial games; Hadamard matrices; Berlekamp's switching game; unbalancing lights problem}

\begin{abstract}
In $1977$, G. Bennett proved, by means of non-deterministic methods, an
inequality which plays a fundamental role in a series of optimization
problems. More precisely, Bennett's inequality shows that, for $p_{1},p_{2}
\in\lbrack1,\infty]$ and all positive integers $n_{1},n_{2}$, there exists a
bilinear form $A_{n_{1},n_{2}}\colon\left(  \mathbb{R}^{n_{1}},\left\Vert
\cdot\right\Vert _{p_{1}}\right)  \times\left(  \mathbb{R}^{n_{2}},\left\Vert
\cdot\right\Vert _{p_{2}}\right)  \longrightarrow\mathbb{R}$ with coefficients
$\pm1$ satisfying
\[
\left\Vert A_{n_{1},n_{2}}\right\Vert \leq C_{p_{1},p_{2}}\max\left\{
n_{1}^{1-\frac{1}{p_{1}}}n_{2}^{\max\left\{  \frac{1}{2}-\frac{1}{p_{2}
},0\right\}  },n_{2}^{1-\frac{1}{p_{2}}}n_{1}^{\max\left\{  \frac{1}{2}
-\frac{1}{p_{1}},0\right\}  }\right\}
\]
for a certain constant $C_{p_{1},p_{2}}$ depending just on $p_{1},p_{2}$;
moreover, the exponents of $n_{1},n_{2}$ cannot be improved. In this paper,
using a constructive approach, we prove that $C_{p_{1},p_{2}}\leq\sqrt{8/5}$
whenever $p_{1},p_{2}\in\left[  2,\infty\right]  $ or $p_{1}=p_{2}=p\in\left[
1,\infty\right]  $. Our techniques are applied to provide new upper bounds for
the constants of a combinatorial game, known as Gale--Berlekamp switching game or unbalancing lights problem. As a consequence, we improve estimates
obtained by Brown and Spencer in $1971$ and by Carlson and Stolarski in $2004$.

\end{abstract}
\maketitle
\tableofcontents

\section{Introduction}

The task of proving sharp inequalities involving combinatorial arguments is
usually rather difficult. A common and powerful approach used in this
framework is to construct a suitable probability space and show that the
probability of the existence of an element with certain properties in this
space is positive. This approach was coined by Paul Erd\"{o}s (see
\cite[Preface]{alon} for a completer account); it is called \textquotedblleft
Probabilistic Method\textquotedblright. One important application of the
Probabilistic Method can be found in the investigation of matrices (or
multilinear forms) with entries $\pm1$ in optimization problems. In this paper
we shall deal with three optimization problems in this framework: Bennett's
inequality, the Kahane--Salem--Zygmund inequality (KSZ inequality for short)
and the Gale--Berlekamp switching game. The Probabilistic Method, although
effective in many aspects, is obviously strongly non constructive. In the
particular case of the KSZ inequality, Bennett's inequality and the
Gale--Berlekamp switching game, it does not provide precise estimates of the
constants involved; in this paper we obtain better constants to these problems
by means of a constructive approach.

If $p_{1},p_{2}\in\lbrack1,\infty]$ and $n_{1},n_{2}$ are positive integers,
let $C_{p_{1},p_{2},n_{1},n_{2}}$ be the smallest constant such that there
exists a bilinear form $A_{n_{1},n_{2}}\colon\ell_{p_{1}}^{n_{1}}\times
\ell_{p_{2}}^{n_{2}}\rightarrow\mathbb{K}$ ($\mathbb{K=R}$ or $\mathbb{C}$)
with coefficients $\pm1$ satisfying%
\begin{equation}
\left\Vert A_{n_{1},n_{2}}\right\Vert \leq C_{p_{1},p_{2},n_{1},n_{2}}%
\max\left\{  n_{1}^{1/p_{1}^{\ast}}n_{2}^{\max\left\{  \frac{1}{2}-\frac
{1}{p_{2}},0\right\}  },n_{2}^{1/p_{2}^{\ast}}n_{1}^{\max\left\{  \frac{1}%
{2}-\frac{1}{p_{1}},0\right\}  }\right\}  \text{.}\label{7u}%
\end{equation}
Above and henceforth $\ell_{p}^{n}$ represents $\mathbb{K}^{n}$ with the
$\ell_{p}$-norm and $p^{\ast}$ denotes the conjugate of $p$, i.e.,
$1/p+1/p^{\ast}=1.$ We also assume $1/\infty=0$ and, as usual, we consider the
$\sup$ norm,
\[
\left\Vert A_{n_{1},n_{2}}\right\Vert :=\sup\left\{  \left\vert A_{n_{1}n_{2}%
}\left(  x,y\right)  \right\vert :\left\Vert x\right\Vert _{\ell_{p_{1}%
}^{n_{1}}}\leq1\text{ and }\left\Vert y\right\Vert _{\ell_{p_{2}}^{n_{1}}}%
\leq1\right\}  \text{.}%
\]
Bennett's inequality shows that%
\begin{equation}
C_{p_{1},p_{2}}:=\sup\left\{  C_{p_{1},p_{2},n_{1},n_{2}}:n_{1},n_{2}%
\in\mathbb{N}\right\}  <\infty\label{9u}%
\end{equation}
and that all exponents in the right-hand-side of (\ref{7u}) are optimal in the
sense that, if we consider smaller exponents in (\ref{7u}), the supremum in
(\ref{9u}) is infinity. The probabilistic techniques used in \cite{be2} do not
provide estimates for $C_{p_{1},p_{2}}$. Variants of this inequality were
independently obtained, also using probabilistic methods, by different authors
(see \cite{bayart,kah,man,mastylo,var}) and nowadays inequalities of this type
are usually called Kahane--Salem--Zygmund inequalities and play a crucial role
in several fields of modern analysis (see \cite{ABPS,israel,be1, boas,
mastylo}).

Recently, we have proved in \cite{JFA} that, for $p_{1},p_{2}\in
\lbrack2,\infty]$, given $\varepsilon>0$, there exists $N$ such that
\[
C_{p_{1},p_{2},n_{1},n_{2}}<1+\varepsilon\text{,}%
\]
whenever $n_{1},n_{2}>N$.

In this paper we continue the investigation initiated in \cite{JFA} and we
show, among other results, that%
\begin{equation}
C_{p_{1},p_{2}}\leq\sqrt{8/5} \label{raiz2}%
\end{equation}
whenever $p_{1},p_{2}\in\left[  2,\infty\right]  $ or $p_{1}=p_{2}\in\left[
1,\infty\right]  $. Using similar techniques, we also obtain new upper bounds
for the constants of a related combinatorial problem, known as Gale--Berlekamp
switching game (sometimes called Berlekamp's switching game or unbalancing
lights problem).

The paper is organized as follows. In Section \ref{Sec2} we show that in order
to prove (\ref{raiz2}) it suffices to prove that $C_{p_{1},p_{2}}\leq
\sqrt{8/5}$ for $p_{1}=p_{2}=2$ and complex scalars. In Section \ref{Sec3} we
prove (\ref{raiz2}). In Section \ref{Sec4} we provide new upper bounds for the
extension of Bennett's inequality for $m$-linear forms $T\colon\ell_{\infty
}^{n_{1}}\times\cdots\times\ell_{\infty}^{n_{m}}\rightarrow\mathbb{K}$. In the
final section (Section \ref{Sec5}) we apply our techniques to obtain new upper
bounds for the constants of the Gale--Berlekamp switching game, improving the
previous best known estimates due to Brown and Spencer (\cite[1971]{TBJS}) and
Carlson and Stolarski (\cite[2004]{car}).

\section{Preliminary results\label{Sec2}}

It is obvious that for $\mathbb{K=R}$ the constants $C_{p_{1},p_{2}%
,n_{1},n_{2}}$ and $C_{p_{1},p_{2}}$ are smaller than or equal to the
respective constants when $\mathbb{K=C}$. In this section we prove some
preliminary results that will help us to prove (\ref{raiz2}) for both real or
complex scalars.

\begin{proposition}
\label{PropA}If $p_{1},p_{2}\in\left[  2,\infty\right]  $, then
\[
C_{p_{1},p_{2}}\leq C_{p,p}\text{ for }2\leq p\leq\min\{p_{1},p_{2}\}
\]
and, in particular, $C_{p_{1},p_{2}}\leq C_{p,p}\leq C_{2,2}$ whenever $2\leq
p\leq\min\{p_{1},p_{2}\}$.
\end{proposition}

\begin{proof}
Let $2\leq p\leq\min\left\{  p_{1},p_{2}\right\}  $, $n_{1},n_{2}\in
\mathbb{N}$ and $A_{0}\colon\ell_{p}^{n_{1}}\times\ell_{p}^{n_{2}}%
\rightarrow\mathbb{K}$ be a bilinear form of the type%
\[
A_{0}\left(  x,y\right)  =%
{\textstyle\sum\limits_{i=1}^{n_{1}}}
{\textstyle\sum\limits_{j=1}^{n_{2}}}
\pm x_{i}y_{j}%
\]
satisfying%
\[
\left\Vert A_{0}\right\Vert \leq C_{p,p,n_{1},n_{2}}\max\left\{
n_{1}^{1/p^{\ast}}n_{2}^{\frac{1}{2}-\frac{1}{p}},n_{2}^{1/p^{\ast}}%
n_{1}^{\frac{1}{2}-\frac{1}{p}}\right\}  \text{.}%
\]

Let $A\colon\ell_{p_{1}}^{n_{1}}\times\ell_{p_{2}}^{n_{2}}\rightarrow
\mathbb{K}$ be the bilinear form defined by%
\[
A\left(  x,y\right)  =A_{0}\left(  x,y\right)  \text{.}%
\]
Note that
\begin{equation}
\left\Vert A\right\Vert =n_{1}^{\frac{1}{p}-\frac{1}{p_{1}}}n_{2}^{\frac{1}%
{p}-\frac{1}{p_{2}}}\sup_{\left\Vert x\right\Vert _{p_{1}},\left\Vert
y\right\Vert _{p_{2}}\leq1}\left\vert
{\textstyle\sum\limits_{i=1}^{n_{1}}}
{\textstyle\sum\limits_{j=1}^{n_{2}}}
\pm n_{1}^{\frac{1}{p_{1}}-\frac{1}{p}}n_{2}^{\frac{1}{p_{2}}-\frac{1}{p}%
}x_{i}y_{j}\right\vert \text{.} \label{B2}%
\end{equation}
Defining $z=n_{1}^{\frac{1}{p_{1}}-\frac{1}{p}}x$ and $w=n_{2}^{\frac{1}%
{p_{2}}-\frac{1}{p}}y$, by the H\"{o}lder inequality we have $\left\Vert
z\right\Vert _{p},\left\Vert w\right\Vert _{p}\leq1$ whenever $\left\Vert
x\right\Vert _{p_{1}},\left\Vert y\right\Vert _{p_{2}}\leq1$. Therefore, by
(\ref{B2}),
\begin{align*}
\left\Vert A\right\Vert  &  \leq n_{1}^{\frac{1}{p}-\frac{1}{p_{1}}}%
n_{2}^{\frac{1}{p}-\frac{1}{p_{2}}}\sup_{\left\Vert z\right\Vert
_{p},\left\Vert w\right\Vert _{p}\leq1}\left\vert
{\textstyle\sum\limits_{i=1}^{n_{1}}}
{\textstyle\sum\limits_{j=1}^{n_{2}}}
\pm z_{i}w_{j}\right\vert \\
&  =n_{1}^{\frac{1}{p}-\frac{1}{p_{1}}}n_{2}^{\frac{1}{p}-\frac{1}{p_{2}}%
}\left\Vert A_{0}\right\Vert \\
&  \leq n_{1}^{\frac{1}{p}-\frac{1}{p_{1}}}n_{2}^{\frac{1}{p}-\frac{1}{p_{2}}%
}C_{p,p,n_{1},n_{2}}\max\left\{  n_{1}^{1/p^{\ast}}n_{2}^{\frac{1}{2}-\frac
{1}{p}},n_{2}^{1/p^{\ast}}n_{1}^{\frac{1}{2}-\frac{1}{p}}\right\} \\
&  =C_{p,p,n_{1},n_{2}}\max\left\{  n_{1}^{1/p_{1}^{\ast}}n_{2}^{\frac{1}%
{2}-\frac{1}{p_{2}}},n_{2}^{1/p_{2}^{\ast}}n_{1}^{\frac{1}{2}-\frac{1}{p_{1}}%
}\right\}  \text{.}%
\end{align*}
This shows that, for every $n_{1},n_{2}\in\mathbb{N}$, we have%
\[
C_{p_{1},p_{2},n_{1},n_{2}}\leq C_{p,p,n_{1},n_{2}}%
\]
and, therefore,
\[
C_{p_{1},p_{2}}\leq C_{p,p}\text{.}%
\]

\end{proof}

\begin{proposition}
\label{PropB}If $\mathbb{K=C}$, then $C_{p,p}\leq C_{2,2}$ for every
$p\in\left[  1,\infty\right]  $.
\end{proposition}

\begin{proof}
If $p\in\left[  2,\infty\right]  $, Proposition \ref{PropA} assures that
$C_{p,p}\leq C_{2,2}$. Let us consider $p\in\lbrack1,2)$ and let $A_{0}%
\colon\ell_{2}^{n_{1}}\times\ell_{2}^{n_{2}}\rightarrow\mathbb{C}$ be a
bilinear form of the type%
\[
A_{0}\left(  x,y\right)  =%
{\textstyle\sum\limits_{i=1}^{n_{1}}}
{\textstyle\sum\limits_{j=1}^{n_{2}}}
\pm x_{i}y_{j}%
\]
such that%
\[
\left\Vert A_{0}\right\Vert \leq C_{2,2,n_{1},n_{2}}\max\left\{  n_{1}%
^{1/2},n_{2}^{1/2}\right\}
\]

It is obvious that if $A_{1}\colon\ell_{1}^{n_{1}}\times\ell_{1}^{n_{2}%
}\rightarrow\mathbb{C}$ is defined by the same rule as $A_{0}$, then
$\left\Vert A_{1}\right\Vert =1$. It follows from the Riesz-Thorin Theorem
that, if the bilinear form $A\colon\ell_{p}^{n_{1}}\times\ell_{p}^{n_{2}%
}\rightarrow\mathbb{C}$ is also defined by the same rule as $A_{0}$, then%
\[
\Vert A\Vert\leq C_{2,2,n_{1},n_{2}}^{2/p^{\ast}}\max\left\{  n_{1}%
^{1/p^{\ast}},n_{2}^{1/p^{\ast}}\right\}  \text{.}%
\]
Finally, since $C_{2,2}\geq1$ (recall that $C_{2,2}\geq C_{2,2,1,1}=1$) and
$2/p^{\ast}\leq1$, we conclude that%
\[
C_{p,p}:=\sup\left\{  C_{p,p,n_{1},n_{2}}:n_{1},n_{2}\in\mathbb{N}\right\}
\leq\sup\left\{  C_{2,2,n_{1},n_{2}}^{2/p^{\ast}}:n_{1},n_{2}\in
\mathbb{N}\right\}  \leq C_{2,2}^{2/p^{\ast}}\leq C_{2,2}\text{.}%
\]

\end{proof}

\begin{proposition}
\label{PropC}If%

\begin{equation}
1\leq p_{1}<2\leq p_{2}\text{ and }n_{1}\leq n_{2} \tag{i}\label{i}%
\end{equation}
or%

\begin{equation}
1\leq p_{2}<2\leq p_{1}\text{ and }n_{1}\geq n_{2}\text{,} \tag{ii}\label{ii}%
\end{equation}

then
\[
C_{p_{1},p_{2},n_{1},n_{2}}\leq C_{2,2}\text{.}%
\]

\end{proposition}

\begin{proof}
Let us prove (\ref{i}); the proof of (\ref{ii}) is analogous.

Note that, in this case,%
\[
\max\left\{  n_{1}^{1/p_{1}^{\ast}}n_{2}^{\max\left\{  \frac{1}{2}-\frac
{1}{p_{2}},0\right\}  },n_{2}^{1/p_{2}^{\ast}}n_{1}^{\max\left\{  \frac{1}%
{2}-\frac{1}{p_{1}},0\right\}  }\right\}  =n_{2}^{1/p_{2}^{\ast}}\text{.}%
\]
Let $A_{0}\colon\ell_{2}^{n_{1}}\times\ell_{p_{2}}^{n_{2}}\rightarrow
\mathbb{K}$ be a bilinear form of the type
\[%
{\textstyle\sum\limits_{i=1}^{n_{1}}}
{\textstyle\sum\limits_{j=1}^{n_{2}}}
\pm x_{i}y_{j}%
\]
such that%
\[
\left\Vert A_{0}\right\Vert \leq C_{2,p_{2},n_{1},n_{2}}\max\left\{
n_{1}^{1/2}n_{2}^{\frac{1}{2}-\frac{1}{p_{2}}},n_{2}^{1/p_{2}^{\ast}}\right\}
=C_{2,p_{2},n_{1},n_{2}}n_{2}^{1/p_{2}^{\ast}}\text{.}%
\]
Let us define $A\colon\ell_{p_{1}}^{n_{1}}\times\ell_{p_{2}}^{n_{2}%
}\rightarrow\mathbb{K}$ by $A\left(  x,y\right)  =A_{0}\left(  x,y\right)  $.
It is obvious that $\left\Vert A\right\Vert \leq\left\Vert A_{0}\right\Vert $
and, from Proposition \ref{PropA},%
\[
\left\Vert A_{0}\right\Vert \leq C_{2,p_{2},n_{1},n_{2}}n_{2}^{1/p_{2}^{\ast}%
}\leq C_{2,p_{2}}n_{2}^{1/p_{2}^{\ast}}\leq C_{2,2}n_{2}^{1/p_{2}^{\ast}%
}\text{.}%
\]

\end{proof}

\begin{proposition}
\label{PropD}If%
\begin{equation}
1\leq p_{1}\leq p_{2}\leq2\text{ and }n_{1}\leq n_{2} \tag{i}\label{aa}%
\end{equation}
or%
\begin{equation}
1\leq p_{2}\leq p_{1}\leq2\text{ and }n_{1}\geq n_{2}\text{,} \tag{ii}%
\label{bb}%
\end{equation}

\noindent then
\[
C_{p_{1},p_{2},n_{1},n_{2}}\leq C_{2,2}\text{.}%
\]

\end{proposition}

\begin{proof}
Let us (again) just prove (\ref{aa}). It is simple to check that%
\[
\max\left\{  n_{1}^{1/p_{1}^{\ast}}n_{2}^{\max\left\{  \frac{1}{2}-\frac
{1}{p_{2}},0\right\}  },n_{2}^{1/p_{2}^{\ast}}n_{1}^{\max\left\{  \frac{1}%
{2}-\frac{1}{p_{1}},0\right\}  }\right\}  =n_{2}^{1/p_{2}^{\ast}}\text{.}%
\]
Let $A_{0}\colon\ell_{p_{2}}^{n_{1}}\times\ell_{p_{2}}^{n_{2}}\rightarrow
\mathbb{K}$ be a bilinear form of the type%
\[
A_{0}\left(  x,y\right)  =%
{\textstyle\sum\limits_{i=1}^{n_{1}}}
{\textstyle\sum\limits_{j=1}^{n_{2}}}
\pm x_{i}y_{j}%
\]
such that%
\[
\left\Vert A_{0}\right\Vert \leq C_{p_{2},p_{2},n_{1},n_{2}}n_{2}%
^{1/p_{2}^{\ast}}\text{.}%
\]
Let us define $A\colon\ell_{p_{1}}^{n_{1}}\times\ell_{p_{2}}^{n_{2}%
}\rightarrow\mathbb{K}$ by $A\left(  x,y\right)  =A_{0}\left(  x,y\right)  $.
Notice that%
\begin{equation}
\left\Vert A\right\Vert \leq\left\Vert A_{0}\right\Vert \leq C_{p_{2}%
,p_{2},n_{1},n_{2}}n_{2}^{1/p_{2}^{\ast}}\leq C_{p_{2},p_{2}}n_{2}%
^{1/p_{2}^{\ast}}\text{.} \label{He}%
\end{equation}
From (\ref{He}) and Proposition \ref{PropB} we conclude that%
\[
C_{p_{1},p_{2},n_{1},n_{2}}\leq C_{p_{2},p_{2}}\leq C_{2,2}\text{.}%
\]

\end{proof}

An immediate consequence of Propositions \ref{PropA}, \ref{PropB}, \ref{PropC}
and \ref{PropD} is:

\begin{corollary}
\label{0aa}If $\mathbb{K=C}$, then
\[
C_{2,2}\geq\left\{
\begin{array}
[c]{ll}%
C_{p_{1},p_{2}}\text{,} & \text{if }p_{1},p_{2}\in\left[  2,\infty\right]
\text{,}\\
& \vspace{-0.3cm}\\
C_{p,p}\text{,} & \text{if }p\in\left[  1,\infty\right]  \text{.}\\
& \vspace{-0.3cm}\\
C_{p_{1},p_{2},n_{1},n_{2}}\text{,} & \text{if }p_{1},p_{2},n_{1},n_{2}\text{
are as in \emph{Propositions \ref{PropC}} or \emph{\ref{PropD}}}\emph{.}%
\end{array}
\right.
\]

\end{corollary}

\section{The constants of Bennett's inequality are bounded by $\sqrt{8/5}%
$\label{Sec3}}

We recall that a Hadamard matrix of order $n$ is a square matrix
$\mathbf{H=}\left[  h_{ij}\right]  _{n\times n}$, with $h_{ij}\in\left\{
-1,1\right\}  $ for all $i,j$, such that
\[
\mathbf{HH}^{\top}=n\mathbf{I}_{n}\text{,}%
\]
where $\mathbf{I}_{n}$ is the identity matrix of order $n$ and $\mathbf{H}%
^{\top}$ is the transpose of $\mathbf{H}$. Thus, if $u_{1},\ldots,u_{n}$ are
the rows of $\mathbf{H}$, then their inner product is
\[
\left\langle u_{i},u_{j}\right\rangle =n\delta_{ij}%
\]
for all $i,j$, where $\delta_{ij}$ denotes the Kronecker delta. It is
well-known that there are Hadamard matrices of order $4k$ with $k$ up to
$166$. It is also known that if $p,q$ are orders of Hadamard matrices, then
$pq$ is also an order of a Hadamard matrix.

\begin{lemma}
\label{LemaFran}For each $n_{1},n_{2}\in\mathbb{N}$, let $r$ be the order of a
Hadamard matrix $\left[  h_{ij}\right]  $ such that $\max\left\{  n_{1}%
,n_{2}\right\}  \leq r$. Then,%
\[
C_{2,2,n_{1},n_{2}}\leq\left(  \frac{r}{\max\left\{  n_{1},n_{2}\right\}
}\right)  ^{1/2}\text{.}%
\]

\end{lemma}

\begin{proof}
Let us define $A\colon\ell_{2}^{n_{1}}\times\ell_{2}^{n_{2}}\rightarrow
\mathbb{K}$ by%
\[
A\left(  x,y\right)  =%
{\textstyle\sum\limits_{i=1}^{n_{1}}}
{\textstyle\sum\limits_{j=1}^{n_{2}}}
h_{ij}x_{i}y_{j}\text{.}%
\]
Let $x=\left(  x_{1},\ldots,x_{n_{1}}\right)  \in\ell_{2}^{n_{1}}$ and
$y=\left(  y_{1},\ldots,y_{n_{2}}\right)  \in\mathbb{\ell}_{2}^{n_{2}}$ be
such that $\left\Vert x\right\Vert _{2}\leq1$ and $\left\Vert y\right\Vert
_{2}\leq1$. If%
\[
\left\{
\begin{array}
[c]{l}%
u=\left(  u_{1},\ldots,u_{r}\right)  =\left(  x_{1},\ldots,x_{n_{1}}%
,0,\ldots0\right)  \text{,}\\
\vspace{-0.3cm}\\
v=\left(  v_{1},\ldots,v_{r}\right)  =\left(  y_{1},\ldots,y_{n_{2}}%
,0,\ldots0\right)  \text{,}%
\end{array}
\right.
\]
then, $\left\Vert u\right\Vert _{2},\left\Vert v\right\Vert _{2}\leq1$ and, by
the Cauchy-Schwarz inequality, we have%
\begin{align*}
\left\vert A\left(  x,y\right)  \right\vert  &  =\left\vert
{\textstyle\sum\limits_{i=1}^{n_{1}}}
\,%
{\textstyle\sum\limits_{j=1}^{n_{2}}}
h_{ij}x_{i}y_{j}\right\vert =\left\vert
{\textstyle\sum\limits_{j=1}^{r}}
\,%
{\textstyle\sum\limits_{i=1}^{r}}
h_{ij}u_{i}v_{j}\right\vert \\
&  \leq\left(
{\textstyle\sum\limits_{j=1}^{r}}
\left\vert v_{j}\right\vert ^{2}\right)  ^{1/2}\left(
{\textstyle\sum\limits_{j=1}^{r}}
\left\vert
{\textstyle\sum\limits_{i=1}^{r}}
h_{ij}u_{i}\right\vert ^{2}\right)  ^{1/2}\\
&  \leq\left(
{\textstyle\sum\limits_{j=1}^{r}}
\,%
{\textstyle\sum\limits_{i,k=1}^{r}}
h_{ij}h_{kj}u_{i}\overline{u_{k}}\right)  ^{1/2}\\
&  =\left(
{\textstyle\sum\limits_{i,k=1}^{r}}
u_{i}\overline{u_{k}}r\delta_{ik}\right)  ^{1/2}\\
&  \leq\left(  \frac{r}{\max\left\{  n_{1},n_{2}\right\}  }\right)  ^{1/2}%
\max\left\{  n_{1}^{1/2},n_{2}^{1/2}\right\}  \text{.}%
\end{align*}

\end{proof}

\begin{remark}
\label{0b}Let $n_{1},n_{2}\in\mathbb{N}$. If $n=\max\left\{  n_{1}%
,n_{2}\right\}  $, then $C_{2,2,n_{1},n_{2}}\leq C_{2,2,n,n}$ and, as a
consequence,
\begin{equation}
\sup\left\{  C_{2,2,n,n}:n\in\mathbb{N}\right\}  =\sup\left\{  C_{2,2,n_{1}%
,n_{2}}:n_{1},n_{2}\in\mathbb{N}\right\}  =C_{2,2}\text{.} \label{megasena}%
\end{equation}
In fact, let $A\colon\ell_{2}^{n}\times\ell_{2}^{n}\rightarrow\mathbb{K}$ a
bilinear form of the type%
\[
A\left(  x,y\right)  =%
{\textstyle\sum\limits_{i=1}^{n}}
{\textstyle\sum\limits_{j=1}^{n}}
\pm x_{i}y_{j}%
\]
such that
\[
\left\Vert A\right\Vert \leq C_{2,2,n,n}n^{1/2}\text{.}%
\]
Considering the bilinear form $A_{0}\colon\ell_{2}^{n_{1}}\times\ell
_{2}^{n_{2}}\rightarrow\mathbb{K}$ given by
\[
A_{0}\left(  x,y\right)  =A\left(  \left(  x_{1},\ldots,x_{n_{1}}%
,0,\ldots,0\right)  ,\left(  y_{1},\ldots,y_{n_{2}},0,\ldots,0\right)
\right)
\]
it is obvious that
\[
\left\Vert A_{0}\right\Vert \leq\left\Vert A\right\Vert \leq C_{2,2,n,n}%
n^{1/2}=C_{2,2,n,n}\max\left\{  n_{1}^{1/2},n_{2}^{1/2}\right\}  .
\]
Hence $C_{2,2,n_{1},n_{2}}\leq C_{2,2,n,n}$ and \emph{(\ref{megasena})}
follows straightforwardly.
\end{remark}

Now we are able to prove the main result of this section.

\begin{theorem}
\label{bbnn}If $p_{1},p_{2},p\in\left[  1,\infty\right]  $, then%
\[
\sqrt{8/5}\geq\left\{
\begin{array}
[c]{ll}%
C_{p_{1},p_{2}}\text{,} & \text{if }p_{1},p_{2}\in\left[  2,\infty\right]
\text{,}\\
& \vspace{-0.3cm}\\
C_{p,p}\text{,} & \text{if }p\in\left[  1,\infty\right]  \text{.}\\
& \vspace{-0.3cm}\\
C_{p_{1},p_{2},n_{1},n_{2}}\text{,} & \text{if }p_{1},p_{2},n_{1},n_{2}\text{
are as in \emph{Propositions \ref{PropC}} or \emph{\ref{PropD}}.}%
\end{array}
\right.
\]
\bigskip
\end{theorem}

\begin{proof}
By Corollary \ref{0aa} and Remark \ref{0b} it is enough to show that
$C_{2,2,n,n}\leq\sqrt{8/5}$ for each $n\in\mathbb{N}$. From Lemma
\ref{LemaFran} we know that%
\[
\left\{
\begin{array}
[c]{l}%
C_{2,2,1,1}=1\text{,}\\
\vspace{-0.3cm}\\
C_{2,2,2,2}\leq1\text{,}\\
\vspace{-0.3cm}\\
C_{2,2,3,3}\leq\left(  4/3\right)  ^{1/2}\approx1.15\text{,}\\
\vspace{-0.3cm}\\
C_{2,2,4,4}\leq1\text{.}%
\end{array}
\right.
\]

We already know that there are Hadamard matrices of order $4k$ with $k$ up to
$166$. For $n=5,\ldots,664$, by Lemma \ref{LemaFran} we have $C_{2,2,n,n}%
\leq1$ whenever $n=4k$. If $4k+1\leq n\leq4k+3$, we have
\[
C_{2,2,n,n}\leq\left(  \dfrac{4\left(  k+1\right)  }{n}\right)  ^{1/2}%
\leq\left(  \dfrac{4\left(  k+1\right)  }{4k+1}\right)  ^{1/2}\text{.}%
\]
It is easy to see that the maximum is achieved when $k=1$ and thus
\[
C_{2,2,n,n}\leq\sqrt{8/5}%
\]
for every $n=1,2,\ldots,664$.

From now on we use the multiplicative property of Hadamard matrices: if there
are Hadamard matrices of orders $p$ and $q$, then there is a Hadamard matrix
of order $pq$. Hence, for each $p\in\left\{  1,\ldots,166\right\}  $ and
$r\in\mathbb{N}$, we know that $2^{r}\cdot4p$ is an order of some Hadamard matrix.

For each $m\in\mathbb{N}$ and each $k\in\left\{  8,\ldots,63\right\}  $ let
$A_{m,k}=\left\{  8^{m}k+1,\ldots,8^{m}\left(  k+1\right)  \right\}  $. Hence,
from Lemma \ref{LemaFran}, if $n\in A_{m,k}$, we have
\[
C_{2,2,n,n}\leq\left(  \dfrac{8^{m}\left(  k+1\right)  }{n}\right)  ^{1/2}%
\leq\left(  \dfrac{8^{m}\left(  k+1\right)  }{8^{m}k}\right)  ^{1/2}=\left(
1+\dfrac{1}{k}\right)  ^{1/2}\leq\left(  1+\dfrac{1}{8}\right)  ^{1/2}%
<\sqrt{8/5}\text{.}%
\]
Let us see that, if $n\geq665$, then $n\in A_{m,k}$ for some $\left(
m,k\right)  \in\mathbb{N}\times\left\{  8,\ldots,63\right\}  $. In fact, if
$n\geq665$, then there exists $m\geq3$ such that%
\[
8^{m-1}\cdot8=8^{m}<n\leq8^{m+1}=8^{m-1}\cdot64
\]
and so, $n\in A_{m-1,8}\cup\cdots\cup A_{m-1,63}$. Therefore,
\[
\mathbb{N=}\left\{  1,\ldots,664\right\}  \cup_{m=2}^{\infty}\cup_{k=8}%
^{63}A_{m,k}%
\]
and the proof is done.
\end{proof}

\section{Improved upper bounds for the constants of the Kahane--Salem--Zygmund
inequality\label{Sec4}}

As we have mentioned at the Introduction, Bennett's inequality is part of a
bulk of similar non-deterministic inequalities that emerged in the
$1970^{\prime}s$ (see \cite{kah,man,var}). However, except for the result of
Bennett, the other approaches have dealt with $n_{1}=\cdots=n_{m}$ and it was
just very recently that an $m$-linear approach with arbitrary $n_{1}%
,...,n_{m}$ appeared (see \cite{alb}): there exists a constant $C_{m}$ such
that, given $p_{1},\ldots,p_{m}\in\left[  1,\infty\right]  $ and $n_{1}%
,\ldots,n_{m}\in\mathbb{N}$, there is an $m$-linear form $A_{n_{1},...,n_{m}%
}\colon\ell_{p_{1}}^{n_{1}}\times\cdots\times\ell_{p_{m}}^{n_{m}}%
\rightarrow\mathbb{K}$ of the type
\[
A_{n_{1},...,n_{m}}\left(  z^{(1)},\ldots,z^{(m)}\right)  =%
{\textstyle\sum\limits_{i_{1}=1}^{n_{1}}}
\cdots%
{\textstyle\sum\limits_{i_{m}=1}^{n_{m}}}
\pm z_{i_{1}}^{(1)}\cdots z_{i_{m}}^{(m)}\text{,}%
\]
such that%
\[
\Vert A_{n_{1},...,n_{m}}\Vert\leq C_{m}\left(
{\textstyle\sum\limits_{k=1}^{m}}
n_{k}\right)  ^{\frac{1}{\min\{\max\{2,p_{1}^{\ast}\},\ldots,\max
\{2,p_{m}^{\ast}\}\}}}%
{\textstyle\prod\limits_{k=1}^{m}}
n_{k}^{\max\left\{  \frac{1}{2}-\frac{1}{p_{k}},0\right\}  }.
\]
Of course, this means that there is a constant $D_{m}\geq C_{m}$ satisfying
\[
\Vert A_{n_{1},...,n_{m}}\Vert\leq D_{m}\max_{k=1,\ldots,m}\left\{
n_{k}^{\frac{1}{\min\{\max\{2,p_{1}^{\ast}\},\ldots,\max\{2,p_{m}^{\ast}\}\}}%
}\right\}
{\textstyle\prod\limits_{k=1}^{m}}
n_{k}^{\max\left\{  \frac{1}{2}-\frac{1}{p_{k}},0\right\}  }\text{.}%
\]
In \cite{alb} the constant $C_{m}$ satisfies
\begin{equation}
C_{m}\leq\sqrt{32m\log\left(  6m\right)  }\sqrt{m!} \label{KSZ1}%
\end{equation}
and in \cite{JFA} we have shown that given $\varepsilon>0$ and a positive
integer $m$, if $p_{1}=\cdots=p_{m}=\infty$, there exists a positive integer
$N$ such that
\[
D_{m}<1+\varepsilon
\]
if we consider $n_{1},...,n_{m}>N$ (and, \textit{a fortiori}, $C_{m}%
<1+\varepsilon$ for $n_{1},...,n_{m}>N$). The next result improves
(\ref{KSZ1}) for the case $p_{1}=\cdots=p_{m}=\infty$ without restrictions on
$n_{1},...,n_{m}$:

\begin{theorem}
\label{TeoKSZ1}For each positive integer $m$ there exists an $m$-linear form
$A_{n_{1},\ldots,n_{m}}\colon\ell_{\infty}^{n_{1}}\times\cdots\times
\ell_{\infty}^{n_{m}}\rightarrow\mathbb{K}$ of the type
\[
A_{n_{1},\ldots,n_{m}}(z^{(1)},\ldots,z^{(m)})=%
{\textstyle\sum\limits_{i_{1}=1}^{n_{1}}}
\cdots%
{\textstyle\sum\limits_{i_{m}=1}^{n_{m}}}
\pm z_{i_{1}}^{(1)}\cdots z_{i_{m}}^{(m)}\text{,}%
\]
satisfying%
\[
\Vert A_{n_{1},\ldots,n_{m}}\Vert\leq2^{\frac{m+1}{2}}\max\left\{  n_{1}%
^{1/2},\ldots,n_{m}^{1/2}\right\}
{\textstyle\prod\limits_{k=1}^{m}}
n_{k}^{1/2}\text{.}%
\]

\end{theorem}

\begin{proof}
Notice that we can assume, without loss of generality, $n_{1}\leq\cdots\leq
n_{m}$. For each $k=1,\ldots,m$, let $t_{k}$ be the unique integer such that
\[
2^{t_{k}}<n_{k}\leq2^{t_{k}+1}%
\]

It follows from \cite[Lemma 2.1]{JFA} that there exists an $m$-linear form
$A_{0}\colon\ell_{\infty}^{2^{t_{1}+1}}\times\cdots\times\ell_{\infty
}^{2^{t_{m}+1}}\rightarrow\mathbb{K}$ with coefficients $\pm1$ such that%
\[
\left\Vert A_{0}\right\Vert \leq\left(  2^{t_{m}+1}\right)  ^{1/2}%
{\textstyle\prod\limits_{k=1}^{m}}
\left(  2^{t_{k}+1}\right)  ^{1/2}%
\]

If $n_{k}=2^{t_{k}+1}$ for all $k=1,\ldots,m$, we just consider $A_{n_{1}%
,\ldots,n_{m}}=A_{0}$. Otherwise, we define%
\begin{align*}
A_{n_{1},\ldots,n_{m}}\colon\ell_{\infty}^{n_{1}}\times\cdots\times
\ell_{\infty}^{n_{m}}  &  \rightarrow\mathbb{K}\\
\left(  \left(  z_{i_{1}}^{\left(  1\right)  }\right)  _{i_{1}=1}^{n_{1}%
},\ldots,\left(  z_{i_{m}}^{\left(  m\right)  }\right)  _{i_{m}=1}^{n_{m}%
}\right)   &  \mapsto A_{0}\left(  \left(  z_{1}^{\left(  1\right)  }%
,\ldots,z_{n_{1}}^{\left(  1\right)  },0,\ldots,0\right)  ,\ldots,\left(
z_{1}^{\left(  m\right)  },\ldots,z_{n_{m}}^{\left(  m\right)  }%
,0,\ldots,0\right)  \right)  \text{.}%
\end{align*}
Then, given $x^{\left(  k\right)  }$ in the closed unit ball of $\ell_{\infty
}^{n_{k}}$, $k=1,\ldots,m$, we have%
\begin{align*}
&  \left\vert A_{n_{1},\ldots,n_{m}}\left(  \left(  x_{i_{1}}^{\left(
1\right)  }\right)  _{i_{1}=1}^{n_{1}},\ldots,\left(  x_{i_{m}}^{\left(
m\right)  }\right)  _{i_{m}=1}^{n_{m}}\right)  \right\vert \\
&  =\left\vert A_{0}\left(  \left(  x_{1}^{\left(  1\right)  },\ldots
,x_{n_{1}}^{\left(  1\right)  },0,\ldots,0\right)  ,\ldots,\left(
x_{1}^{\left(  m\right)  },\ldots,x_{n_{m}}^{\left(  m\right)  }%
,0,\ldots,0\right)  \right)  \right\vert \\
&  \leq\left\Vert A_{0}\right\Vert \leq\left(  2^{t_{m}+1}\right)  ^{1/2}%
{\textstyle\prod\limits_{k=1}^{m}}
\left(  2^{t_{k}+1}\right)  ^{1/2}\\
&  =\frac{\left(  2^{t_{m}+1}\right)  ^{1/2}%
{\textstyle\prod\limits_{k=1}^{m}}
\left(  2^{t_{k}+1}\right)  ^{1/2}}{\max\left\{  n_{1}^{1/2},\ldots
,n_{m}^{1/2}\right\}
{\textstyle\prod\limits_{k=1}^{m}}
n_{k}^{1/2}}\max\left\{  n_{1}^{1/2},\ldots,n_{m}^{1/2}\right\}
{\textstyle\prod\limits_{j=1}^{m}}
n_{j}^{1/2}\\
&  \leq\frac{\left(  2^{t_{m}+1}\right)  ^{1/2}%
{\textstyle\prod\limits_{k=1}^{m}}
\left(  2^{t_{k}+1}\right)  ^{1/2}}{\left(  2^{t_{m}}\right)  ^{1/2}%
{\textstyle\prod\limits_{j=1}^{m}}
\left(  2^{t_{k}}\right)  ^{1/2}}\max\left\{  n_{1}^{1/2},\ldots,n_{m}%
^{1/2}\right\}
{\textstyle\prod\limits_{j=1}^{m}}
n_{j}^{1/2}\\
&  =2^{\frac{m+1}{2}}\max\left\{  n_{1}^{1/2},\ldots,n_{m}^{1/2}\right\}
{\textstyle\prod\limits_{j=1}^{m}}
n_{j}^{1/2}%
\end{align*}
and this ends the proof.
\end{proof}

Now we show that, if we consider the particular case $n_{1}=\cdots=n_{m}=n$,
the estimates given by Theorem \ref{TeoKSZ1} can be improved to $\sqrt{\left(
8/5\right)  ^{m-1}}$. For this purpose, we shall adapt some arguments and
techniques from the preceding sections.

For each $n_{1},\ldots,n_{m}\in\mathbb{N}$, let $C_{\infty,\ldots,\infty
,n_{1},\ldots,n_{m}}$ be the smallest of the constants $C$ for which there is
an $m$-linear form $A\colon\ell_{\infty}^{n_{1}}\times\cdots\times\ell
_{\infty}^{n_{m}}\rightarrow\mathbb{K}$ of the type%
\[
A(x^{(1)},...,x^{(m)})=%
{\textstyle\sum\limits_{i_{1}=1}^{n_{1}}}
\cdots%
{\textstyle\sum\limits_{i_{m}=1}^{n_{m}}}
\pm x_{i_{1}}^{(1)}\cdots x_{i_{m}}^{(m)}%
\]
such that%
\[
\Vert A\Vert\leq C\max\left\{  n_{1}^{1/2},\ldots,n_{m}^{1/2}\right\}
{\textstyle\prod\limits_{k=1}^{m}}
n_{k}^{1/2}\text{.}%
\]
Similarly, for each $n_{1},\ldots,n_{m}\in\mathbb{N}$, let $C_{2,\infty
,\ldots,\infty,2,n_{1},\ldots,n_{m}}$ be the smallest of the constants $C$ for
which there is an $m$-linear form $A\colon\ell_{2}^{n_{1}}\times\ell_{\infty
}^{n_{2}}\times\cdots\times\ell_{\infty}^{n_{m-1}}\times\ell_{2}^{n_{m}%
}\rightarrow\mathbb{K}$ of the type%
\[
A(x^{(1)},...,x^{(m)})=%
{\textstyle\sum\limits_{i_{1}=1}^{n_{1}}}
\cdots%
{\textstyle\sum\limits_{i_{m}=1}^{n_{m}}}
\pm x_{i_{1}}^{(1)}\cdots x_{i_{m}}^{(m)}%
\]
such that%
\[
\Vert A\Vert\leq C\max\left\{  n_{1}^{1/2},\ldots,n_{m}^{1/2}\right\}
{\textstyle\prod\limits_{k=2}^{m-1}}
n_{k}^{1/2}\text{.}%
\]

\begin{proposition}
\label{LemaKSZ1}Let $m$ be a positive integer and let $f\colon\mathbb{N}%
^{m}\rightarrow\lbrack0,\infty)$ be a function. If for each $\left(
n_{1},\ldots,n_{m}\right)  \in\mathbb{N}^{m}$ there exists an $m$-linear form
$A_{0}\colon\ell_{2}^{n_{1}}\times\ell_{\infty}^{n_{2}}\times\cdots\times
\ell_{\infty}^{n_{m-1}}\times\ell_{2}^{n_{m}}\rightarrow\mathbb{K}$ of the
type%
\begin{equation}
A_{0}\left(  x^{\left(  1\right)  },\ldots,x^{\left(  m\right)  }\right)  =%
{\textstyle\sum\limits_{i_{1}=1}^{n_{1}}}
\cdots%
{\textstyle\sum\limits_{i_{m}=1}^{n_{m}}}
\pm x_{i_{1}}^{\left(  1\right)  }\cdots x_{i_{m}}^{\left(  m\right)  }
\label{KSZ2}%
\end{equation}
such that%
\[
\left\Vert A_{0}\right\Vert \leq f\left(  n_{1},\ldots,n_{m}\right)  \text{,}%
\]
then, for each $\left(  n_{1},\ldots,n_{m}\right)  \in\mathbb{N}^{m}$, there
is an $m$-linear form $A\colon\ell_{\infty}^{n_{1}}\times\cdots\times
\ell_{\infty}^{n_{m}}\rightarrow\mathbb{K}$ as in \emph{(\ref{KSZ2})} such
that%
\[
\left\Vert A\right\Vert \leq\left(  n_{1}n_{m}\right)  ^{1/2}f\left(
n_{1},\ldots,n_{m}\right)  \text{.}%
\]

\end{proposition}

\begin{proof}
Let $A\colon\ell_{\infty}^{n_{1}}\times\cdots\times\ell_{\infty}^{n_{m}%
}\rightarrow\mathbb{K}$ be the $m$-linear form defined by%
\[
A\left(  x^{\left(  1\right)  },\ldots,x^{\left(  m\right)  }\right)
=A_{0}\left(  x^{\left(  1\right)  },\ldots,x^{\left(  m\right)  }\right)
\text{.}%
\]
Hence%
\begin{align}
\left\Vert A\right\Vert  &  =\sup_{\left\Vert x^{\left(  1\right)
}\right\Vert _{\infty},\ldots,\left\Vert x^{\left(  m\right)  }\right\Vert
_{\infty}\leq1}\left\vert
{\textstyle\sum\limits_{i_{1}=1}^{n_{1}}}
\cdots%
{\textstyle\sum\limits_{i_{m}=1}^{n_{m}}}
\pm x_{i_{1}}^{\left(  1\right)  }\cdots x_{i_{m}}^{\left(  m\right)
}\right\vert \nonumber\\
&  =\left(  n_{1}n_{m}\right)  ^{1/2}\sup_{\left\Vert x^{\left(  1\right)
}\right\Vert _{\infty},\ldots,\left\Vert x^{\left(  m\right)  }\right\Vert
_{\infty}\leq1}\left\vert
{\textstyle\sum\limits_{i_{1}=1}^{n_{1}}}
\cdots%
{\textstyle\sum\limits_{i_{m}=1}^{n_{m}}}
\pm n_{1}^{-1/2}x_{i_{1}}^{\left(  1\right)  }x_{i_{2}}^{\left(  2\right)
}\cdots x_{i_{m-1}}^{\left(  m-1\right)  }n_{m}^{-1/2}x_{i_{m}}^{\left(
m\right)  }\right\vert \text{.} \label{KSZ3}%
\end{align}
For $k\in\left\{  1,m\right\}  $, let $z^{\left(  k\right)  }=n_{k}%
^{-1/2}x^{\left(  k\right)  }$. From the H\"{o}lder inequality, since
$\left\Vert x^{\left(  k\right)  }\right\Vert _{\infty}\leq1$, we have
\[
\left\Vert z^{\left(  k\right)  }\right\Vert _{2}\leq1
\]
and, therefore, from (\ref{KSZ3}),
\begin{align*}
\left\Vert A\right\Vert  &  \leq\left(  n_{1}n_{m}\right)  ^{1/2}%
\sup_{\left\Vert z^{\left(  1\right)  }\right\Vert _{2},\left\Vert x^{\left(
2\right)  }\right\Vert _{\infty},\ldots,\left\Vert x^{\left(  m-1\right)
}\right\Vert _{\infty},\left\Vert z^{\left(  m\right)  }\right\Vert _{2}\leq
1}\left\vert
{\textstyle\sum\limits_{i_{1}=1}^{n_{1}}}
\cdots%
{\textstyle\sum\limits_{i_{m}=1}^{n_{m}}}
\pm z_{i_{1}}^{\left(  1\right)  }x_{i_{2}}^{\left(  2\right)  }\cdots
x_{i_{m-1}}^{\left(  m-1\right)  }z_{i_{m}}^{\left(  m\right)  }\right\vert \\
&  =\left(  n_{1}n_{m}\right)  ^{1/2}\left\Vert A_{0}\right\Vert \\
&  \leq\left(  n_{1}n_{m}\right)  ^{1/2}f\left(  n_{1},\ldots,n_{m}\right)
\text{,}%
\end{align*}
as we wished.
\end{proof}

\begin{proposition}
\label{PropKSZ1}Let $n_{1},\ldots,n_{m}\in\mathbb{N}$ be such that $n_{1}%
=\min\left\{  n_{1},n_{2},\ldots,n_{m}\right\}  $. Then%
\[
C_{\infty,\ldots,\infty,n_{1},\ldots,n_{m}}\leq C_{2,\infty,\ldots
,\infty,2,n_{1},\ldots,n_{m}}\text{.}%
\]

\end{proposition}

\begin{proof}
Let $f\colon\mathbb{N}^{m}\rightarrow\lbrack0,\infty)$ be the function defined
by
\[
f\left(  n_{1},\ldots,n_{m}\right)  =C_{2,\infty,\ldots,\infty,2,n_{1}%
,\ldots,n_{m}}\max\left\{  n_{1}^{1/2},\ldots,n_{m}^{1/2}\right\}
{\textstyle\prod\limits_{k=2}^{m-1}}
n_{k}^{1/2}\text{.}%
\]
We know that, in this case, for each $\left(  n_{1},\ldots,n_{m}\right)
\in\mathbb{N}^{m}$ there is an $m$-linear form $A_{0}\colon\ell_{2}^{n_{1}%
}\times\ell_{\infty}^{n_{2}}\times\cdots\times\ell_{\infty}^{n_{m-1}}%
\times\ell_{2}^{n_{m}}\rightarrow\mathbb{K}$ of the type%
\[
A_{0}\left(  x^{\left(  1\right)  },\ldots,x^{\left(  m\right)  }\right)  =%
{\textstyle\sum\limits_{i_{1}=1}^{n_{1}}}
\cdots%
{\textstyle\sum\limits_{i_{m}=1}^{n_{m}}}
\pm x_{i_{1}}^{\left(  1\right)  }\cdots x_{i_{m}}^{\left(  m\right)  }%
\]
such that%
\[
\left\Vert A_{0}\right\Vert =C_{2,\infty,\ldots,\infty,2,n_{1},\ldots,n_{m}%
}\max\left\{  n_{1}^{1/2},\ldots,n_{m}^{1/2}\right\}
{\textstyle\prod\limits_{k=2}^{m-1}}
n_{k}^{1/2}\text{.}%
\]
From Lemma \ref{LemaKSZ1}, there is an $m$-linear form $A_{1}\colon
\ell_{\infty}^{n_{1}}\times\cdots\times\ell_{\infty}^{n_{m}}\rightarrow
\mathbb{K}$ with coefficients $\pm1$ such that%
\begin{align}
\left\Vert A_{1}\right\Vert  &  \leq\left(  n_{1}n_{m}\right)  ^{1/2}%
C_{2,\infty,\ldots,\infty,2,n_{1},\ldots,n_{m}}\max\left\{  n_{1}^{1/2}%
,\ldots,n_{m}^{1/2}\right\}
{\textstyle\prod\limits_{k=2}^{m-1}}
n_{k}^{1/2}\nonumber\\
&  =C_{2,\infty,\ldots,\infty,2,n_{1},\ldots,n_{m}}\max\left\{  n_{1}%
^{1/2},\ldots,n_{m}^{1/2}\right\}
{\textstyle\prod\limits_{k=1}^{m}}
n_{k}^{1/2} \label{KSZ4}%
\end{align}
But, by the definition of $C_{\infty,\ldots,\infty,n_{1},\ldots,n_{m}}$, it is
obvious that%
\begin{equation}
C_{\infty,\ldots,\infty,n_{1},\ldots,n_{m}}\max\left\{  n_{1}^{1/2}%
,\ldots,n_{m}^{1/2}\right\}
{\textstyle\prod\limits_{k=1}^{m}}
n_{k}^{1/2}\leq\left\Vert A_{1}\right\Vert \text{.} \label{KSZ5}%
\end{equation}
Hence, from (\ref{KSZ4}) and (\ref{KSZ5}), we obtain%
\[
C_{\infty,\ldots,\infty,n_{1},\ldots,n_{m}}\leq C_{2,\infty,\ldots
,\infty,2,n_{1},\ldots,n_{m}}%
\]
and the result follows.
\end{proof}

\begin{proposition}
\label{PropKSZ2}Let $n_{1},n_{2},\ldots,n_{m}$ be positive integers such that
$n_{1}=\min\left\{  n_{1},n_{2},\ldots,n_{m}\right\}  $, $n_{m}=\max\left\{
n_{1},n_{2},\ldots,n_{m}\right\}  $ and, for each $k=2,\ldots,m$ there exists
a Hadamard matrix $\mathbf{H}_{k}=\left[  h_{ij}^{\left(  k\right)  }\right]
_{n_{k}\times n_{k}}$. Then, the $m$-linear form%
\[
A_{0}\colon\ell_{2}^{n_{1}}\times\ell_{\infty}^{n_{2}}\times\cdots\times
\ell_{\infty}^{n_{m-1}}\times\ell_{2}^{n_{m}}\rightarrow\mathbb{C}%
\]
defined by%
\[
A_{0}\left(  x^{(1)},\dots,x^{(m)}\right)  =%
{\textstyle\sum\limits_{i_{1}=1}^{n_{1}}}
\cdots%
{\textstyle\sum\limits_{i_{m}=1}^{n_{m}}}
h_{i_{1}i_{2}}^{(2)}h_{i_{2}i_{3}}^{(3)}\cdots h_{i_{m-1}i_{m}}^{(m)}x_{i_{1}%
}^{(1)}\cdots x_{i_{m}}^{(m)}%
\]
is such that%
\[
\left\Vert A_{0}\right\Vert \leq n_{m}^{1/2}%
{\textstyle\prod\limits_{k=2}^{m-1}}
n_{k}^{1/2}=\max\left\{  n_{1}^{1/2},\ldots,n_{m}^{1/2}\right\}
{\textstyle\prod\limits_{k=2}^{m-1}}
n_{k}^{1/2}\text{.}%
\]

\end{proposition}

\begin{proof}
For each $k=2,\ldots,m$, let $u_{i}^{\left(  k\right)  }$, $i=1,\ldots,n_{k}$,
be the $i$-th rows of $\mathbf{H}_{k}$. Hence
\begin{equation}
\left\langle u_{i}^{\left(  k\right)  },u_{j}^{\left(  k\right)
}\right\rangle =n_{k}\delta_{ij}\text{.} \label{KSZ6}%
\end{equation}
Let us consider the square matrices of order $n_{m}$ defined by%
\[
\left[  \mathfrak{h}_{ij}^{\left(  k\right)  }\right]  _{n_{m}\times n_{m}%
}:=\left[
\begin{array}
[c]{cc}%
\mathbf{H}_{k} & \mathbf{0}_{n_{k}\times\left(  n_{m}-n_{k}\right)  }\\
\mathbf{0}_{\left(  n_{m}-n_{k}\right)  \times n_{k}} & \mathbf{0}_{\left(
n_{m}-n_{k}\right)  \times\left(  n_{m}-n_{k}\right)  }%
\end{array}
\right]
\]
for each $k=2,\ldots,m$ and, given $x^{(k)}\in\ell_{p_{k}}^{n_{k}}$, where
$p_{k}=2$ if $k\in\left\{  1,m\right\}  $ and $p_{k}=\infty$ if $k\in\left\{
2,\ldots,m-1\right\}  $, let us denote
\[
y^{(k)}=\left(  x_{1}^{(k)},\ldots,x_{n_{k}}^{(k)},0,\ldots,0\right)  \in
\ell_{p_{k}}^{n_{m}}\text{.}%
\]
Let us define%
\[
A_{0}\colon\ell_{2}^{n_{1}}\times\ell_{\infty}^{n_{2}}\cdots\times\ell
_{\infty}^{n_{m-1}}\times\ell_{2}^{n_{m}}\rightarrow\mathbb{K}%
\]
by
\[
A_{0}\left(  x^{(1)},\dots,x^{(m)}\right)  =%
{\textstyle\sum\limits_{i_{1}=1}^{n_{1}}}
\cdots%
{\textstyle\sum\limits_{i_{m}=1}^{n_{m}}}
\left(
{\textstyle\prod\limits_{k=2}^{m}}
h_{i_{k-1}i_{k}}^{(k)}\right)  \left(
{\textstyle\prod\limits_{k=1}^{m}}
x_{i_{k}}^{(k)}\right)  ={%
{\textstyle\sum\limits_{i_{1},\dots,i_{m}=1}^{n_{m}}}
}\left(
{\textstyle\prod\limits_{k=2}^{m}}
\mathfrak{h}_{i_{k-1}i_{k}}^{(k)}\right)  \left(
{\textstyle\prod\limits_{k=1}^{m}}
y_{i_{k}}^{(k)}\right)
\]
and notice that
\[%
{\textstyle\prod\limits_{k=2}^{m}}
h_{i_{k-1}i_{k}}^{(k)}\in\left\{  -1,1\right\}
\]
whenever $i_{j}\in\{1,\ldots,n_{j}\}$ for all $j=1,\ldots,m$.

Assuming that, $x^{\left(  k\right)  }$ rests in the closed unit ball
$B_{\ell_{p_{k}}^{n_{k}}}$ for each each $k=1,\ldots,m$, we have $y^{\left(
k\right)  }$ in the closed unit ball $B_{\ell_{p_{k}}^{n_{m}}}$ for each
$k=1,\ldots,m$ and, by the H\"{o}lder inequality,%
\begin{align*}
\left\vert A_{0}\left(  x^{(1)},\dots,x^{(m)}\right)  \right\vert  &
=\left\vert
{\textstyle\sum\limits_{i_{1},\dots,i_{m}=1}^{n_{m}}}
\left(
{\textstyle\prod\limits_{k=2}^{m}}
\mathfrak{h}_{i_{k-1}i_{k}}^{(k)}\right)  \left(
{\textstyle\prod\limits_{k=1}^{m}}
y_{i_{k}}^{(k)}\right)  \right\vert \\
&  \leq%
{\textstyle\sum\limits_{i_{m}=1}^{n_{m}}}
\left\vert
{\textstyle\sum\limits_{i_{1},\dots,i_{m-1}=1}^{n_{m}}}
\left(
{\textstyle\prod\limits_{k=2}^{m}}
\mathfrak{h}_{i_{k-1}i_{k}}^{(k)}\right)  \left(
{\textstyle\prod\limits_{k=1}^{m-1}}
y_{i_{k}}^{(k)}\right)  \right\vert \left\vert y_{i_{m}}^{\left(  m\right)
}\right\vert \\
&  \leq\left(
{\textstyle\sum\limits_{i_{m}=1}^{n_{m}}}
\left\vert
{\textstyle\sum\limits_{i_{1},\dots,i_{m-1}=1}^{n_{m}}}
\left(
{\textstyle\prod\limits_{k=2}^{m}}
\mathfrak{h}_{i_{k-1}i_{k}}^{(k)}\right)  \left(
{\textstyle\prod\limits_{k=1}^{m-1}}
y_{i_{k}}^{(k)}\right)  \right\vert ^{2}\right)  ^{1/2}\left\Vert y^{\left(
m\right)  }\right\Vert _{2}\\
&  \leq\left(
{\textstyle\sum\limits_{i_{m}=1}^{n_{m}}}
\left\vert
{\textstyle\sum\limits_{i_{1},\dots,i_{m-1}=1}^{n_{m}}}
\left(
{\textstyle\prod\limits_{k=2}^{m}}
\mathfrak{h}_{i_{k-1}i_{k}}^{(k)}\right)  \left(
{\textstyle\prod\limits_{k=1}^{m-1}}
y_{i_{k}}^{(k)}\right)  \right\vert ^{2}\right)  ^{1/2}\\
&  =\left(
{\textstyle\sum\limits_{i_{m}=1}^{n_{m}}}
{\textstyle\sum\limits_{\substack{i_{1},\dots,i_{m-1}=1\\j_{1},\dots
,j_{m-1}=1}}^{n_{m}}}
\left(
{\textstyle\prod\limits_{k=2}^{m-1}}
\mathfrak{h}_{i_{k-1}i_{k}}^{(k)}\mathfrak{h}_{j_{k-1}j_{k}}^{(k)}\right)
\mathfrak{h}_{i_{m-1}i_{m}}^{(m)}\mathfrak{h}_{j_{m-1}i_{m}}^{(m)}\left(
{\textstyle\prod\limits_{k=1}^{m-1}}
y_{i_{k}}^{(k)}\overline{y_{j_{k}}^{(k)}}\right)  \right)  ^{1/2}\text{.}%
\end{align*}
Hence,
\begin{align*}
\left\vert A_{0}\left(  x^{(1)},\dots,x^{(m)}\right)  \right\vert  &
\leq\left(
{\textstyle\sum\limits_{\substack{i_{1},\dots,i_{m-1}=1\\j_{1},\dots
,j_{m-1}=1}}^{n_{m}}}
\left(
{\textstyle\prod\limits_{k=2}^{m-1}}
\mathfrak{h}_{i_{k-1}i_{k}}^{(k)}\mathfrak{h}_{j_{k-1}j_{k}}^{\left(
k\right)  }\right)  \left(
{\textstyle\prod\limits_{k=1}^{m-1}}
y_{i_{k}}^{(k)}\overline{y_{j_{k}}^{(k)}}\right)
{\textstyle\sum\limits_{i_{m}=1}^{n_{m}}}
\mathfrak{h}_{i_{m-1}i_{m}}^{(m)}\mathfrak{h}_{j_{m-1}i_{m}}^{(m)}\right)
^{1/2}\\
&  =\left(
{\textstyle\sum\limits_{\substack{i_{1},\dots,i_{m-1}=1\\j_{1},\dots
,j_{m-1}=1}}^{n_{m}}}
\left(
{\textstyle\prod\limits_{k=2}^{m-1}}
\mathfrak{h}_{i_{k-1}i_{k}}^{(k)}\mathfrak{h}_{j_{k-1}j_{k}}^{\left(
k\right)  }\right)  \left(
{\textstyle\prod\limits_{k=1}^{m-1}}
y_{i_{k}}^{(k)}\overline{y_{j_{k}}^{(k)}}\right)  \left\langle u_{i_{m-1}%
}^{\left(  m\right)  },u_{j_{m-1}}^{\left(  m\right)  }\right\rangle \right)
^{1/2}%
\end{align*}
and, from (\ref{KSZ6}), we have
\begin{align*}
&  \left\vert A_{0}\left(  x^{(1)},\dots,x^{(m)}\right)  \right\vert \\
&  \leq\left(
{\textstyle\sum\limits_{i_{m-1}=1}^{n_{m}}}
{\textstyle\sum\limits_{\substack{i_{1},\dots,i_{m-2}=1\\j_{1},\dots
,j_{m-2}=1}}^{n_{m}}}
\left(
{\textstyle\prod\limits_{k=2}^{m-2}}
\mathfrak{h}_{i_{k-1}i_{k}}^{(k)}\mathfrak{h}_{j_{k-1}j_{k}}^{\left(
k\right)  }\right)  \left(
{\textstyle\prod\limits_{k=1}^{m-2}}
y_{i_{k}}^{(k)}\overline{y_{j_{k}}^{(k)}}\right)  \mathfrak{h}_{i_{m-2}%
i_{m-1}}^{(m-1)}\mathfrak{h}_{j_{m-2}i_{m-1}}^{(m-1)}\left\vert y_{i_{m-1}%
}^{(m-1)}\right\vert ^{2}n_{m}\right)  ^{1/2}\\
&  \leq n_{m}^{1/2}\left(
{\textstyle\sum\limits_{\substack{i_{1},\dots,i_{m-2}=1\\j_{1},\dots
,j_{m-2}=1}}^{n_{m}}}
\left(
{\textstyle\prod\limits_{k=2}^{m-2}}
\mathfrak{h}_{i_{k-1}i_{k}}^{(k)}\mathfrak{h}_{j_{k-1}j_{k}}^{\left(
k\right)  }\right)  \left(
{\textstyle\prod\limits_{k=1}^{m-2}}
y_{i_{k}}^{(k)}\overline{y_{j_{k}}^{(k)}}\right)
{\textstyle\sum\limits_{i_{m-1}=1}^{n_{m}}}
\mathfrak{h}_{i_{m-2}i_{m-1}}^{(m-1)}\mathfrak{h}_{j_{m-2}i_{m-1}}%
^{(m-1)}\right)  ^{1/2}\\
&  =n_{m}^{1/2}\left(
{\textstyle\sum\limits_{\substack{i_{1},\dots,i_{m-2}=1\\j_{1},\dots
,j_{m-2}=1}}^{n_{m}}}
\left(
{\textstyle\prod\limits_{k=2}^{m-2}}
\mathfrak{h}_{i_{k-1}i_{k}}^{(k)}\mathfrak{h}_{j_{k-1}j_{k}}^{\left(
k\right)  }\right)  \left(
{\textstyle\prod\limits_{k=1}^{m-2}}
y_{i_{k}}^{(k)}\overline{y_{j_{k}}^{(k)}}\right)  \left\langle u_{i_{m-2}%
}^{\left(  m-1\right)  },u_{j_{m-2}}^{\left(  m-1\right)  }\right\rangle
\right)  ^{1/2}\text{.}%
\end{align*}
Since%
\begin{align*}
&  \left(
{\textstyle\sum\limits_{\substack{i_{1},\dots,i_{m-2}=1\\j_{1},\dots
,j_{m-2}=1}}^{n_{m}}}
\left(
{\textstyle\prod\limits_{k=2}^{m-2}}
\mathfrak{h}_{i_{k-1}i_{k}}^{(k)}\mathfrak{h}_{j_{k-1}j_{k}}^{\left(
k\right)  }\right)  \left(
{\textstyle\prod\limits_{k=1}^{m-2}}
y_{i_{k}}^{(k)}\overline{y_{j_{k}}^{(k)}}\right)  \left\langle u_{i_{m-2}%
}^{\left(  m-1\right)  },u_{j_{m-2}}^{\left(  m-1\right)  }\right\rangle
\right)  ^{1/2}\\
&  =n_{m-1}^{1/2}\left(
{\textstyle\sum\limits_{i_{m-2}=1}^{n_{m}}}
{\textstyle\sum\limits_{\substack{i_{1},\dots,i_{m-3}=1\\j_{1},\dots
,j_{m-3}=1}}^{n_{m}}}
\left(
{\textstyle\prod\limits_{k=2}^{m-3}}
\mathfrak{h}_{i_{k-1}i_{k}}^{(k)}\mathfrak{h}_{j_{k-1}j_{k}}^{\left(
k\right)  }\right)  \left(
{\textstyle\prod\limits_{k=1}^{m-3}}
y_{i_{k}}^{(k)}\overline{y_{j_{k}}^{(k)}}\right)  \mathfrak{h}_{i_{m-3}%
i_{m-2}}^{(m-2)}\mathfrak{h}_{j_{m-3}i_{m-2}}^{(m-2)}\left\vert y_{i_{m-2}%
}^{(m-2)}\right\vert ^{2}\right)  ^{1/2}\\
&  \leq n_{m-1}^{1/2}\left(
{\textstyle\sum\limits_{i_{m-2}=1}^{n_{m}}}
{\textstyle\sum\limits_{\substack{i_{1},\dots,i_{m-3}=1\\j_{1},\dots
,j_{m-3}=1}}^{n_{m}}}
\left(
{\textstyle\prod\limits_{k=2}^{m-3}}
\mathfrak{h}_{i_{k-1}i_{k}}^{(k)}h_{j_{k-1}j_{k}}^{\left(  k\right)  }\right)
\left(
{\textstyle\prod\limits_{k=1}^{m-3}}
y_{i_{k}}^{(k)}\overline{y_{j_{k}}^{(k)}}\right)  \mathfrak{h}_{i_{m-3}%
i_{m-2}}^{(m-2)}\mathfrak{h}_{j_{m-3}i_{m-2}}^{(m-2)}\right)  ^{1/2}\\
&  =n_{m-1}^{1/2}\left(
{\textstyle\sum\limits_{\substack{i_{1},\dots,i_{m-3}=1\\j_{1},\dots
,j_{m-3}=1}}^{n_{m}}}
\left(
{\textstyle\prod\limits_{k=2}^{m-3}}
\mathfrak{h}_{i_{k-1}i_{k}}^{(k)}\mathfrak{h}_{j_{k-1}j_{k}}^{\left(
k\right)  }\right)  \left(
{\textstyle\prod\limits_{k=1}^{m-3}}
y_{i_{k}}^{(k)}\overline{y_{j_{k}}^{(k)}}\right)
{\textstyle\sum\limits_{i_{m-2}=1}^{n_{m}}}
\mathfrak{h}_{i_{m-3}i_{m-2}}^{(m-2)}\mathfrak{h}_{j_{m-3}i_{m-2}}%
^{(m-2)}\right)  ^{1/2}\\
&  =n_{m-1}^{1/2}\left(
{\textstyle\sum\limits_{\substack{i_{1},\dots,i_{m-3}=1\\j_{1},\dots
,j_{m-3}=1}}^{N}}
\left(
{\textstyle\prod\limits_{k=2}^{m-3}}
\mathfrak{h}_{i_{k-1}i_{k}}^{(k)}\mathfrak{h}_{j_{k-1}j_{k}}^{\left(
k\right)  }\right)  \left(
{\textstyle\prod\limits_{k=1}^{m-3}}
y_{i_{k}}^{(k)}\overline{y_{j_{k}}^{(k)}}\right)  \left\langle u_{i_{m-2}%
}^{\left(  m-2\right)  },u_{j_{m-2}}^{\left(  m-2\right)  }\right\rangle
\right)  ^{1/2}\text{,}%
\end{align*}
we have%
\begin{align*}
&  \left\vert A_{0}\left(  x^{(1)},\dots,x^{(m)}\right)  \right\vert \\
&  \leq n_{m}^{1/2}n_{m-1}^{1/2}\left(
{\textstyle\sum\limits_{\substack{i_{1},\dots,i_{m-3}=1\\j_{1},\dots
,j_{m-3}=1}}^{n_{m}}}
\left(
{\textstyle\prod\limits_{k=2}^{m-3}}
\mathfrak{h}_{i_{k-1}i_{k}}^{(k)}\mathfrak{h}_{j_{k-1}j_{k}}^{\left(
k\right)  }\right)  \left(
{\textstyle\prod\limits_{k=1}^{m-3}}
y_{i_{k}}^{(k)}\overline{y_{j_{k}}^{(k)}}\right)  \left\langle u_{i_{m-2}%
}^{\left(  m-2\right)  },u_{j_{m-2}}^{\left(  m-2\right)  }\right\rangle
\right)  ^{1/2}%
\end{align*}
and, repeating this procedure, we finally infer that%
\begin{align*}
\left\vert A_{0}\left(  x^{(1)},\dots,x^{(m)}\right)  \right\vert  &  \leq
n_{m}^{1/2}n_{m-1}^{1/2}\cdots n_{2}^{1/2}\left(
{\textstyle\sum\limits_{i_{1}=1}^{n_{m}}}
\left\vert y_{i_{1}}^{(1)}\right\vert ^{2}\right)  ^{1/2}\\
&  \leq n_{m}^{1/2}%
{\textstyle\prod\limits_{k=2}^{m-1}}
n_{k}^{1/2}%
\end{align*}
and the proof is done.
\end{proof}

\begin{corollary}
\label{CorKSZ2}Let $n_{1},n_{2},\ldots,n_{m}\in\mathbb{N}$ be such that
$n_{1}=\min\left\{  n_{1},n_{2},\ldots,n_{m}\right\}  $ and $n_{m}%
=\max\left\{  n_{1},n_{2},\ldots,n_{m}\right\}  $. Let $r_{1}=n_{1}$ and, for
each $k=2,\ldots,m$ let $r_{k}\geq n_{k}$ be a positive integer for which
there exists a Hadamard matrix of order $r_{k}$, with $r_{m}=\max\left\{
r_{1},\ldots,r_{m}\right\}  $. Then,%
\[
C_{2,\infty,\ldots,\infty,2,n_{1},\ldots,n_{m}}\leq\left(  r_{m}/n_{m}\right)
^{1/2}%
{\textstyle\prod\limits_{k=2}^{m-1}}
\left(  r_{k}/n_{k}\right)  ^{1/2}\text{.}%
\]

\end{corollary}

\begin{proof}
It follows from Proposition \ref{PropKSZ2} that there is an $m$-linear form
$A\colon\ell_{2}^{r_{1}}\times\ell_{\infty}^{r_{2}}\times\cdots\times
\ell_{\infty}^{r_{m-1}}\times\ell_{2}^{r_{m}}\rightarrow\mathbb{K}$ of the
type%
\[
A(z^{(1)},...,z^{(m)})=%
{\textstyle\sum\limits_{i_{1}=1}^{r_{1}}}
\cdots%
{\textstyle\sum\limits_{i_{m}=1}^{r_{m}}}
\pm z_{i_{1}}^{(1)}\cdots z_{i_{m}}^{(m)}%
\]
such that%
\[
\left\Vert A\right\Vert \leq r_{m}^{1/2}%
{\textstyle\prod\limits_{k=2}^{m-1}}
r_{k}^{1/2}\text{.}%
\]
If $A_{0}$ is the restriction of $A$ to $\ell_{2}^{n_{1}}\times\ell_{\infty
}^{n_{2}}\times\cdots\times\ell_{\infty}^{n_{m-1}}\times\ell_{2}^{n_{m}}$ then
$\left\Vert A_{0}\right\Vert \leq\left\Vert A\right\Vert $ and the result follows.
\end{proof}

\begin{theorem}
For each $m\in\mathbb{N}$, there holds%
\[
C_{\infty,\ldots,\infty}\leq\left(  8/5\right)  ^{\frac{m-1}{2}}\text{.}%
\]

\end{theorem}

\begin{proof}
Notice that, from Proposition \ref{PropKSZ1} and Corollary \ref{CorKSZ2},%
\[
\left\{
\begin{array}
[c]{l}%
C_{\infty,\ldots,\infty,1,\ldots,1}=1\text{,}\\
\vspace{-0.3cm}\\
C_{\infty,\ldots,\infty,2,\ldots,2}\leq C_{2,\infty,\ldots,\infty
,2,2,\ldots,2}\leq1\text{,}\\
\vspace{-0.3cm}\\
C_{\infty,\ldots,\infty,3,\ldots,3}\leq C_{2,\infty,\ldots,\infty
,2,3,\ldots,3}\leq\left(  4/3\right)  ^{\frac{m-1}{2}}\text{,}\\
\vspace{-0.3cm}\\
C_{\infty,\ldots,\infty,4,\ldots,4}\leq C_{2,\infty,\ldots,\infty
,2,4,\ldots,4}\leq1\text{.}\\
\vspace{-0.3cm}\\
C_{\infty,\ldots,\infty,5,\ldots,5}\leq C_{2,\infty,\ldots,\infty
,2,5,\ldots,5}\leq\left(  8/5\right)  ^{\frac{m-1}{2}}\\
\vspace{-0.3cm}\\
C_{\infty,\ldots,\infty,6\ldots,6}\leq C_{2,\infty,\ldots,\infty,2,6,\ldots
,6}\leq\left(  8/6\right)  ^{\frac{m-1}{2}}\\
\vspace{-0.3cm}\\
C_{\infty,\ldots,\infty,7,\ldots,7}\leq C_{2,\infty,\ldots,\infty
,2,7,\ldots,7}\leq\left(  8/7\right)  ^{\frac{m-1}{2}}\\
\vspace{-0.3cm}\\
C_{\infty,\ldots,\infty,8,\ldots,8}\leq C_{2,\infty,\ldots,\infty
,2,8,\ldots,8}\leq1
\end{array}
\right.
\]
For $9\leq n\leq64$, let $k$ be the unique positive integer such that
\begin{equation}
4k<n\leq4\left(  k+1\right)  \text{.} \label{KSZ7}%
\end{equation}
From Proposition \ref{PropKSZ1} and Corollary \ref{CorKSZ2}, as well as
(\ref{KSZ7}),
\begin{align*}
C_{\infty,\ldots,\infty,n\ldots,n}  &  \leq C_{2,\infty,\ldots,\infty
,2,n,\ldots,n}\leq\left(  \frac{4\left(  k+1\right)  }{n}\right)  ^{\frac
{m-1}{2}}\\
&  =\left(  \frac{4\left(  k+1\right)  }{4k}\right)  ^{\frac{m-1}{2}}=\left(
1+\frac{1}{k}\right)  ^{\frac{m-1}{2}}\leq\left(  3/2\right)  ^{\frac{m-1}{2}}%
\end{align*}

If $n\geq65$, according with the concepts and notations established in Theorem
\ref{bbnn}, there exists $\left(  t,k\right)  \in\mathbb{N}\times\left\{
8,\ldots,63\right\}  $ such that $n\in A_{t,k}$, that is,%
\begin{equation}
8^{t}k<n\leq8^{t}\left(  k+1\right)  \text{,} \label{KSZ8}%
\end{equation}
for some $t\in\mathbb{N}$ and some $8\leq k\leq63$. From Proposition
\ref{PropKSZ1} and Corollary \ref{CorKSZ2}, as well as (\ref{KSZ8}),
\begin{align*}
C_{\infty,\ldots,\infty,n\ldots,n}  &  \leq C_{2,\infty,\ldots,\infty
,2,n,\ldots,n}\leq\left(  \frac{8^{t}\left(  k+1\right)  }{n}\right)
^{\frac{m-1}{2}}\\
&  =\left(  \frac{8^{t}\left(  k+1\right)  }{8^{t}k}\right)  ^{\frac{m-1}{2}%
}=\left(  1+\frac{1}{k}\right)  ^{\frac{m-1}{2}}\leq\left(  9/8\right)
^{\frac{m-1}{2}}\text{.}%
\end{align*}
This ends the proof.
\end{proof}

\section{The Gale--Berlekamp switching game\label{Sec5}}

The Gale-Berlekamp switching game or unbalancing lights problem (see
\cite[Section 2.5]{alon} and \cite[Chapter 6]{spencer}) consists of an
$n\times n$ square matrix of light bulbs set up at an initial configuration
$\Theta_{n}$. The board has $n$ row and $n$ column switches, which invert the
on-off state of each bulb (on to off and off to on) in the corresponding row
or column. Let $i(\Theta_{n})$ denote the smallest final number of on-lights
achievable by row and column switches starting from $\Theta_{n}$. The goal is
to find $i(\Theta_{n}^{(0)})$ when $\Theta_{n}^{(0)}$ is (one of) the worst
initial pattern, i.e., $i(\Theta_{n}^{(0)})\geq i(\Theta_{n})$ for all
$\Theta_{n}$. Thus we want to estimate
\[
R_{n}:=\max\{i(\Theta_{n}):\Theta_{n}\text{ is an initial configuration of
}n\times n\text{ lights}\}\text{.}%
\]
Sometimes the problem is posed as to find the maximum of the difference
between the state of the light bulbs (starting from (one of) the worst initial
patterns, as before), which we shall henceforth denote by $G_{n}$. It is
simple to check that
\begin{equation}
R_{n}=\frac{1}{2}\left(  n^{2}-G_{n}\right)  \text{.} \label{Ber}%
\end{equation}
The determination of the exact value of $R_{n}$ seems to be conceivable only
for small values of $n$ due to involving combinatorial arguments. The exact
value of $R_{n}$ for $n$ up to $12$ was obtained by Carlson and Stolarski
(\cite{car}; see also \cite{TBJS,slo}). For bigger values of $n$, optimal
constructive approaches seem impracticable and no algorithm to construct such
a \textquotedblleft bad\textquotedblright\ configuration $\Theta_{n}$ seems to
be known. Thus, for bigger values of $n$, probabilistic (non-deterministic)
methods are used to provide estimates for $R_{n}$ and $G_{n}$. The natural
approach to modeling Berlekamp's switching game is by associating $+1$ to the
on-lights and $-1$ to the off-lights from the array of lights $\left(
a_{ij}\right)  _{i,j=1}^{n}$ and observing that
\[
G_{n}=\min\left\{  \max_{\left(  x_{i}\right)  _{i=1}^{n},\left(
y_{j}\right)  _{j=1}^{n}\in\{-1,1\}^{n}}\left\vert
{\textstyle\sum\limits_{i,j=1}^{n}}
a_{ij}x_{i}y_{j}\right\vert :a_{ij}=-1\text{ or }+1\right\}  \text{,}%
\]
where $x_{i}$ and $y_{j}$ denote the switches of row $i$ and of column $j$, respectively.%

\begin{multicols}{2}%
%

\begin{figure}[H]
\centering\includegraphics[width=8cm]{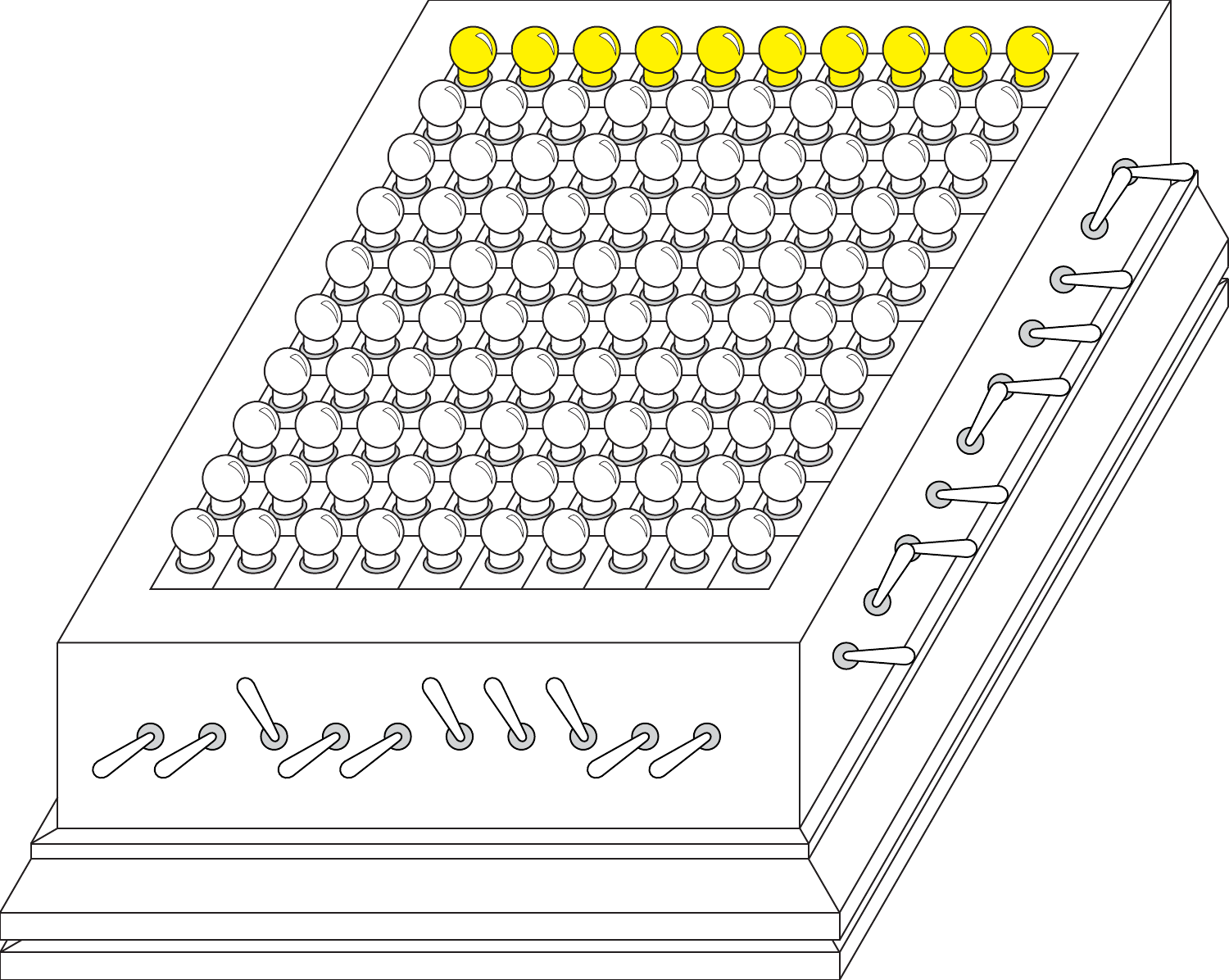}
\caption{An $10 \times
10$ matrix of light bulbs in an initial configuration $\Theta_{10}%
$ such that $i(\Theta_{10}) = 0$.}
\end{figure}%
%

\begin{figure}[H]
\centering\includegraphics[width=8cm]{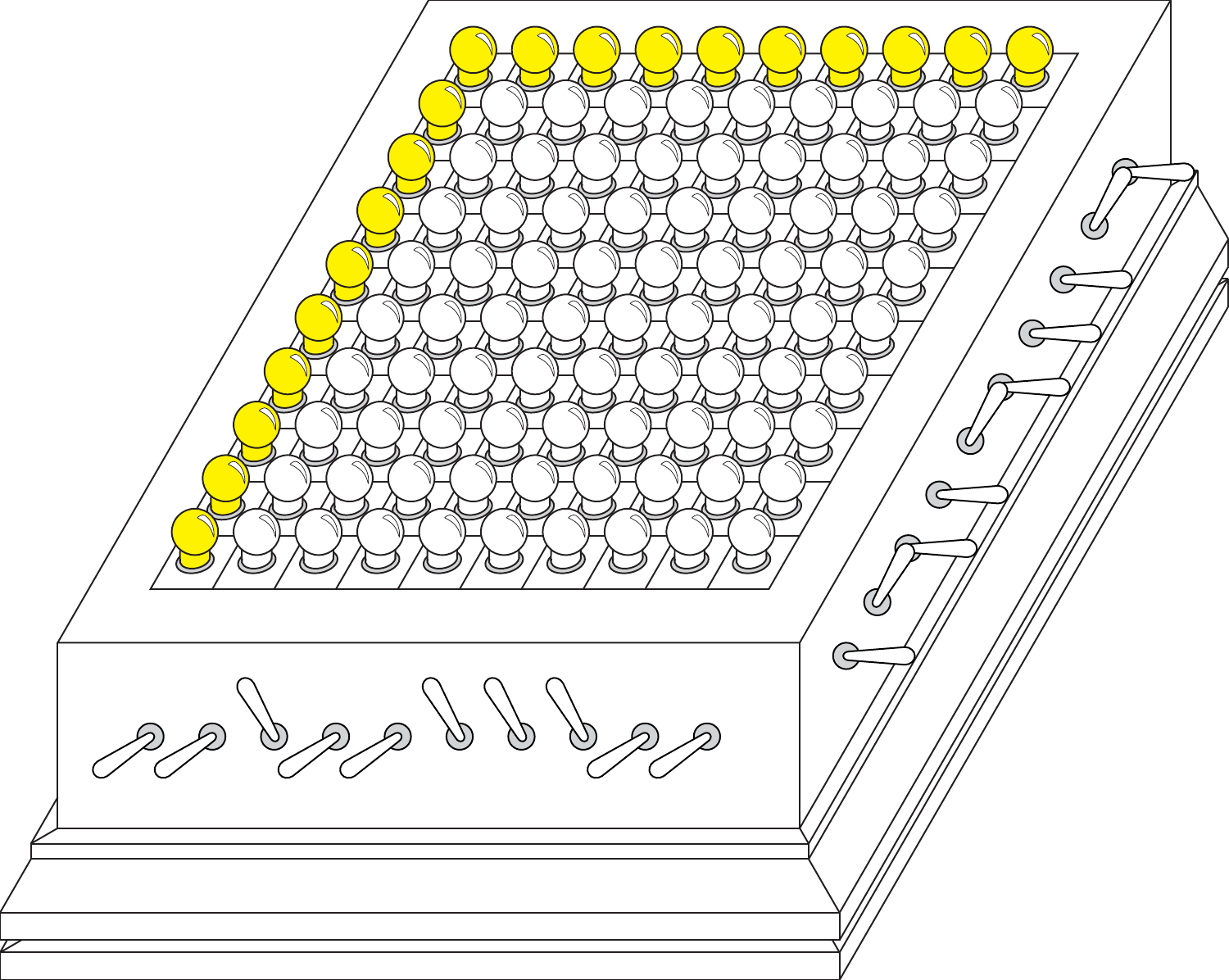}
\caption{An $10 \times
10$ matrix of light bulbs in an initial configuration $\Theta_{10}%
$ such that $i(\Theta_{10}) = 1$.}
\end{figure}%
%

\end{multicols}%

It is already known that $G_{n}$ behaves asymptotically as $n^{3/2}$. More
precisely (see \cite{JFA}),%
\[
\sqrt{\frac{2}{\pi}}+o\left(  1\right)  \leq\dfrac{G_{n}}{n^{3/2}}%
\leq1+o\left(  1\right)
\]
and, in a non-asymptotic viewpoint, it is simple to observe that
\[
\frac{1}{\sqrt{2}}\leq\dfrac{G_{n}}{n^{3/2}}=C_{\infty,\infty,n,n}\text{,}%
\]
where $C_{\infty,\infty,n,n}$ denotes the constant of Bennett's inequality for
$p_{1}=p_{2}=\infty$ and $n_{1}=n_{2}=n$ and the lower bound was obtained by
Littlewood's $4/3$ inequality with constant $\sqrt{2}$ (see \cite[Theorem
3.1]{ara}). By Theorem \ref{bbnn}, we obtain the following estimate for
$G_{n}/n^{3/2}$:
\begin{equation}
\frac{1}{\sqrt{2}}\leq\dfrac{G_{n}}{n^{3/2}}\leq\sqrt{8/5} \label{665544}%
\end{equation}
for all positive integers $n$. In this section we improve the upper bound
$\sqrt{8/5}$ in (\ref{665544}) and improve the best known upper estimates for
$R_{n}$ and $G_{n}$ for $n=15,17,18,19$.

Since $G_{n}\leq n^{3/2}$ (see \cite[Lemma 2.1]{JFA}) whenever $n$ is the
order of a Hadamard matrix, then%
\begin{equation}
R_{n}\geq\dfrac{1}{2}\left(  n^{2}-n\sqrt{n}\right)  \label{1}%
\end{equation}
for all orders $n$ of Hadamard matrices (the above inequality appears in
\cite[page 168]{TBJS}). In the next theorem we use the techniques from Section
\ref{Sec3} to obtain an extension of (\ref{1}) and provide better estimates
for the cases $n=15,17,18,19$, improving the respective estimates from Table
\ref{table1} below.%

\begin{table}[H]
\centering\caption{Estimates for $R_{n}$.}
\smallskip\begin{tabular}
[c]{|c|c|c|c|}\hline
$n$ & Brown-Spencer, $1971$ & Fishburn-Sloane, $1989$ & Carlson-Stolarski,
$2004$\\\hline$1$ & $=0$ & $=0$ & $=0$\\\hline$2$ & $=1$ & $=1$ & $=1$\\\hline
$3$ & $=2$ & $=2$ & $=2$\\\hline$4$ & $=4$ & $=4$ & $=4$\\\hline
$5$ & $=7$ & $=7$ & $=7$\\\hline$6$ & $-$ & $=11$ & $=11$\\\hline
$7$ & $-$ & $=16$ & $=16$\\\hline$8$ & $\geq22$ & $=22$ & $=22$\\\hline
$9$ & $-$ & $=27$ & $=27$\\\hline$10$ & $\geq32$ & $-$ & $=35$\\\hline
$11$ & $-$ & $-$ & $=43$\\\hline$12$ & $-$ & $-$ & $=54$\\\hline
$13$ & $-$ & $-$ & $\geq60$\\\hline$14$ & $-$ & $-$ & $\geq71$\\\hline
$15$ & $\geq72$ & $-$ & $\geq82$\\\hline$16$ & $\geq96$ & $-$ & $\geq
94$\\\hline$17$ & $-$ & $-$ & $\geq106$\\\hline$18$ & $-$ & $-$ & $\geq
120$\\\hline$19$ & $-$ & $-$ & $\geq132$\\\hline$20$ & $\geq156$ & $-$ & $\geq
148$\\\hline\end{tabular}
\label{table1}
\end{table}%

\begin{theorem}
$R_{15}\geq83$, $R_{17}\geq107$, $R_{18}\geq122$, $R_{19}\geq139$.
\end{theorem}

\begin{proof}
Given a positive integer $n$, let $k_{n}$ be the smallest order of a Hadamard
matrix $\left[  h_{ij}\right]  _{k_{n}\times k_{n}}$ such that%
\[
n\leq k_{n}\text{.}%
\]
Obviously,%
\begin{equation}
G_{n}\leq\max_{x_{i},y_{j}\in\{-1,1\}}\left\vert
{\textstyle\sum\limits_{i=1}^{n}}
{\textstyle\sum\limits_{j=1}^{n}}
h_{ij}x_{i}y_{j}\right\vert \text{.} \label{ABAB}%
\end{equation}
Let $z=\left(  z_{1},\ldots,z_{n}\right)  ,w=\left(  w_{1},\ldots
,w_{n}\right)  \in\mathbb{\ell}_{2}^{n}$ be such that $\left\Vert z\right\Vert
_{2},\left\Vert w\right\Vert _{2}\leq1$. Proceeding as in Lemma \ref{LemaFran}%
, we conclude that%
\[
\left\vert
{\textstyle\sum\limits_{i=1}^{n}}
\,%
{\textstyle\sum\limits_{j=1}^{n}}
h_{ij}z_{i}w_{j}\right\vert \leq\sqrt{k_{n}}\text{.}%
\]
So, considering $f\colon\mathbb{N}\rightarrow\mathbb{R}$ given by $f\left(
n\right)  =\sqrt{k_{n}}$, it follows from \cite[Lemma 3.1]{JFA} that
\begin{equation}
\left\vert
{\textstyle\sum\limits_{i=1}^{n}}
\,%
{\textstyle\sum\limits_{j=1}^{n}}
h_{ij}x_{i}y_{j}\right\vert \leq n\sqrt{k_{n}}\text{.} \label{BCBC}%
\end{equation}
whenever $\left\Vert x\right\Vert _{\infty},\left\Vert y\right\Vert _{\infty
}\leq1$. Denoting $G_{2,n}=n\sqrt{k_{n}}$ and combining (\ref{ABAB}) and
(\ref{BCBC}) we have%
\[
G_{n}\leq G_{2,n}\text{.}%
\]

Remembering that, for each $n\in\mathbb{N}$,%
\[
R_{n}=\dfrac{1}{2}\left(  n^{2}-G_{n}\right)
\]
we have%
\begin{equation}
R_{n}\geq\dfrac{1}{2}\left(  n^{2}-G_{2,n}\right)  \text{.} \label{A1}%
\end{equation}
Hence, since%
\begin{align*}
G_{2,15}  &  =15\sqrt{16}=60\\
G_{2,17}  &  =17\sqrt{20}\approx76.02\\
G_{2,18}  &  =18\sqrt{20}\approx80.49\\
G_{2,19}  &  =19\sqrt{20}\approx84.97
\end{align*}
we conclude from (\ref{A1}) that%
\[
R_{15}\geq83\text{, \ \ \ }R_{17}\geq107\text{, \ \ \ \ }R_{18}\geq122\text{,
\ \ \ \ }R_{19}\geq139\text{,}%
\]
as we wished.
\end{proof}

%

\begin{table}[H]
\caption{New estimates for $R_{n}$ when $n=15,17,18,19$}
\smallskip\begin{tabularx}{\textwidth}
{|r|X|X|X|X|}\hline$n$ & \centering$15$ & \centering$17$ & \centering
$18$ & $\frac{}{}$ \hfill$19$ \hfill$\frac{}{}$\\
\hline$R_{n}$ & \centering$\geq83$ & \centering$\geq107$ & \centering
$\geq122$ & $\frac{}{}$ \hfill$\geq139$ \hfill$\frac{}{}$\\
\hline\end{tabularx}
\label{table2}
\end{table}%

We finish this paper combining the results of this section and Section
\ref{Sec3} to improve the upper bound of (\ref{665544}) and provide better
constants to Bennett's inequality when $p_{1}=p_{2}=\infty$. It follows from
(\ref{Ber}) that
\[
G_{n}=n^{2}-2R_{n}\text{.}%
\]
By the first case of Proposition \ref{PropA} we know that
\begin{equation}
C_{\infty,\infty,n,n}\leq C_{2,2,n,n}\text{ for each }n\in\mathbb{N}\text{.}
\label{AK}%
\end{equation}
So, using Tables \ref{table1} and \ref{table2}, (\ref{AK}) and Lemma
\ref{LemaFran} we have%

\begin{table}[H]%
%

\centering
%

\caption{Upper bounds for $C_{\infty,n}$ with $1 \leq n \leq20$.}%
\smallskip%
\begin{tabular}
[c]{|c|c|c|c|c|c|c|c|c|}\cline{1-4}\cline{6-9}%
$n$ & $R_{n}$ & $G_{n}$ & $C_{\infty,\infty,n,n}=G_{n}/n^{3/2}$ &  & $n$ &
$R_{n}$ & $G_{n}$ & $C_{\infty,\infty,n,n}=G_{n}/n^{3/2}$\\\cline{1-4}%
\cline{6-9}%
$1$ & $=0$ & $=1$ & $=1$ &  & $11$ & $=43$ & $=35$ & $=35/11^{3/2}%
<0.96$\\\cline{1-4}\cline{6-9}%
$2$ & $=1$ & $=2$ & $=2/2^{3/2}<0.71$ &  & $12$ & $=54$ & $=36$ &
$=36/12^{3/2}<0.87$\\\cline{1-4}\cline{6-9}%
$3$ & $=2$ & $=5$ & $=5/3^{3/2}<0.97$ &  & $13$ & $\geq60$ & $\leq49$ &
$\leq49/13^{3/2}<1.05$\\\cline{1-4}\cline{6-9}%
$4$ & $=4$ & $=8$ & $=8/4^{3/2}=1$ &  & $14$ & $\geq71$ & $\leq54$ &
$\leq54/14^{3/2}<1.04$\\\cline{1-4}\cline{6-9}%
$5$ & $=7$ & $=11$ & $=11/5^{3/2}<0.99$ &  & $15$ & $\geq83$ & $\leq59$ &
$\leq59/15^{3/2}<1.02$\\\cline{1-4}\cline{6-9}%
$6$ & $=11$ & $=14$ & $=14/6^{3/2}<0.96$ &  & $16$ & $-$ & $-$ & $\leq
1$\\\cline{1-4}\cline{6-9}%
$7$ & $=16$ & $=17$ & $=17/7^{3/2}<0.92$ &  & $17$ & $\geq107$ & $\leq75$ &
$\leq75/17^{3/2}<1.08$\\\cline{1-4}\cline{6-9}%
$8$ & $=22$ & $=20$ & $=20/8^{3/2}<0.89$ &  & $18$ & $\geq122$ & $\leq80$ &
$\leq80/18^{3/2}<1.05$\\\cline{1-4}\cline{6-9}%
$9$ & $=27$ & $=27$ & $=27/9^{3/2}=1$ &  & $19$ & $\geq139$ & $\leq83$ &
$\leq83/19^{3/2}<1.01$\\\cline{1-4}\cline{6-9}%
$10$ & $=35$ & $=30$ & $=30/10^{3/2}<0.95$ &  & $20$ & $-$ & $-$ & $\leq
1$\\\cline{1-4}\cline{6-9}%
\end{tabular}
%

\label{table3}%
%

\end{table}%

\begin{theorem}
For all positive integers $n$, we have%
\[
0.70\approx\frac{1}{\sqrt{2}}\leq\dfrac{G_{n}}{n^{3/2}}\leq\frac{75\sqrt{17}%
}{289}\approx1.07\text{.}%
\]

\end{theorem}

\begin{proof}
All that it is left to prove is the upper estimate. It follows from Table
\ref{table3} that is enough to consider the case $n>20$. Since the upper bound
provided by Table \ref{table3} for $n=1,\ldots,20$ is $75\sqrt{17}/289$, it
suffices to prove that the same upper bound holds for $n>20$.

If $21\leq n\leq664$, let $k\in\left\{  5,\ldots,165\right\}  $ be such that
$4k<n\leq4\left(  k+1\right)  $. So, by (\ref{AK}) and Lemma \ref{LemaFran},
we have%
\[
\dfrac{G_{n}}{n^{3/2}}=C_{\infty,\infty,n,n}\leq C_{2,2,n,n}\leq\left(
\frac{4\left(  k+1\right)  }{n}\right)  ^{1/2}<\left(  \frac{4\left(
k+1\right)  }{4k+1}\right)  ^{1/2}\leq\sqrt{\frac{24}{21}}\text{.}%
\]
Now, for each $m\in\mathbb{N}$ and each $k\in\left\{  8,\ldots,63\right\}  $
let $A_{m,k}$ as in Theorem \ref{bbnn}. Fixed $n\geq665$, we know that there
is $\left(  m,k\right)  \in\mathbb{N}\times\left\{  8,\ldots,63\right\}  $
such that $n\in A_{m,k}$ and, therefore, by (\ref{AK}) and Lemma
\ref{LemaFran}, we obtain%
\[
\dfrac{G_{n}}{n^{3/2}}=C_{\infty,\infty,n,n}\leq C_{2,2,n,n}\leq\left(
\frac{8^{m}\left(  k+1\right)  }{n}\right)  ^{1/2}<\left(  \frac{8^{m}\left(
k+1\right)  }{8^{m}k}\right)  ^{1/2}\leq\sqrt{\frac{9}{8}}\text{.}%
\]
Since $\sqrt{24/21}$ and $\sqrt{9/8}$ are smaller than $75\sqrt{17}/289$, the
proof is completed.
\end{proof}

%

\begin{figure}[H]
\centering\includegraphics[width=5.9cm]{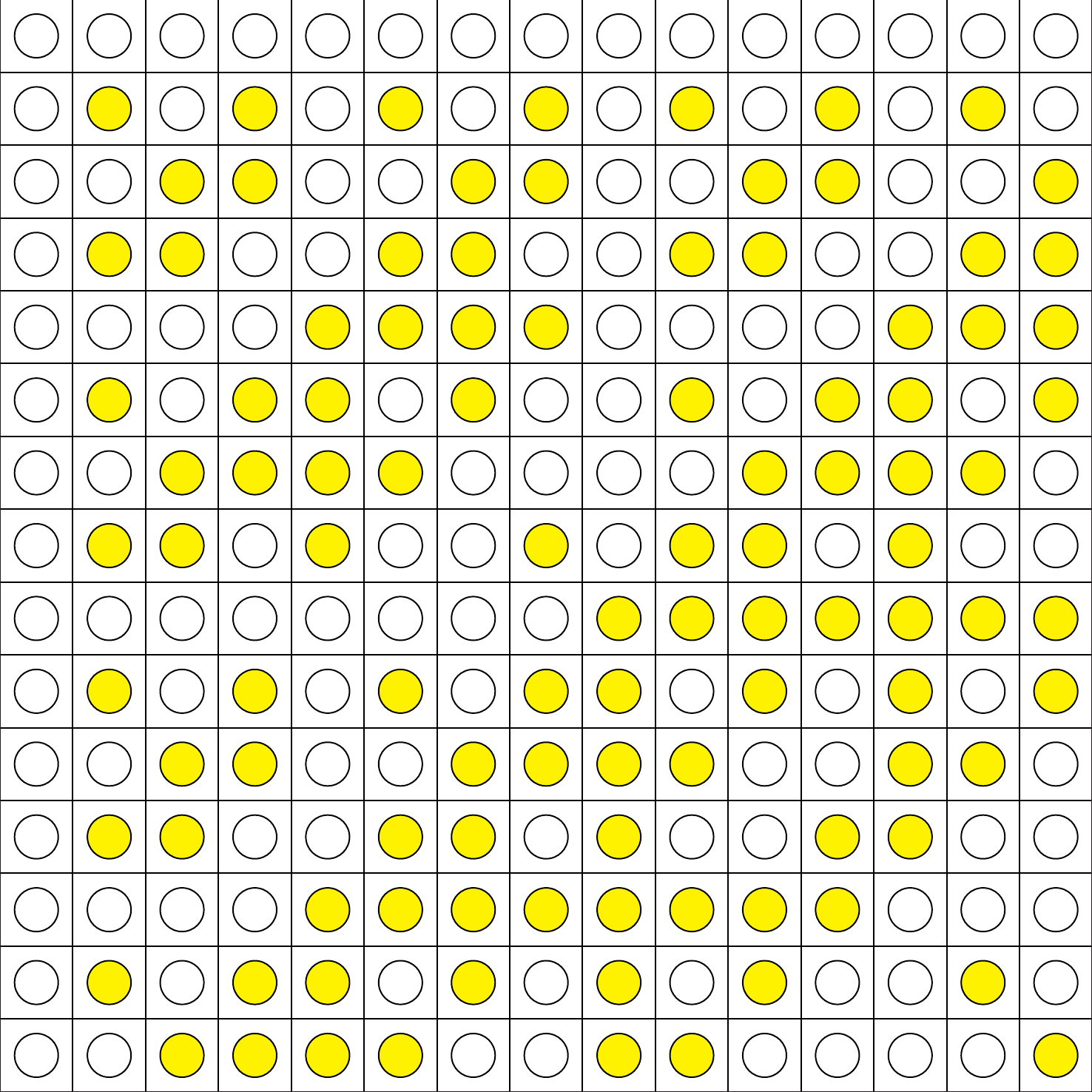}
\caption{Initial configuration of lights generating $R_{15} \geq83$}
\end{figure}%
%

\begin{figure}[H]
\centering\includegraphics[width=5.9cm]{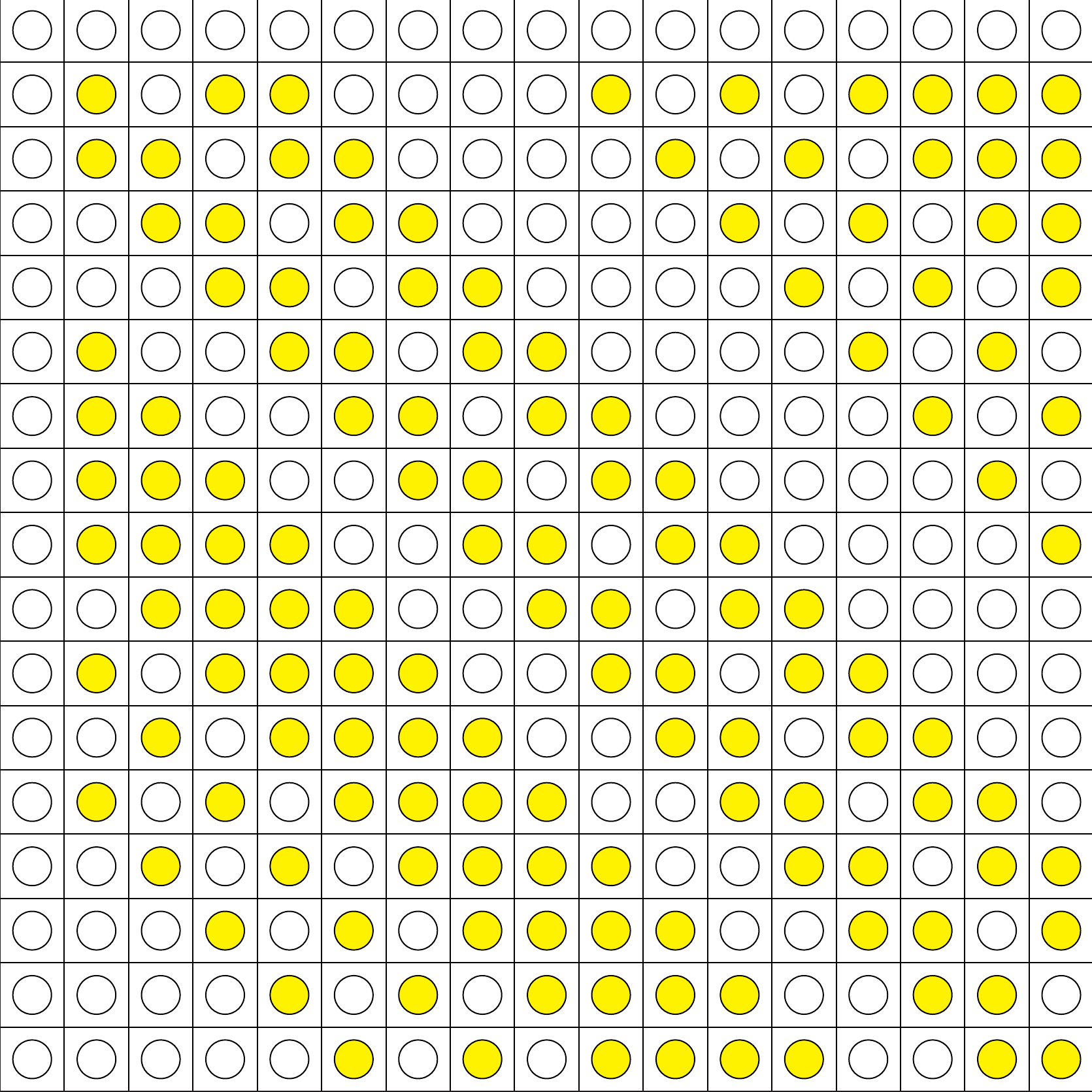}
\caption{Initial configuration of lights generating $R_{17} \geq107$}
\end{figure}%
%

\begin{figure}[H]
\centering\includegraphics[width=5.9cm]{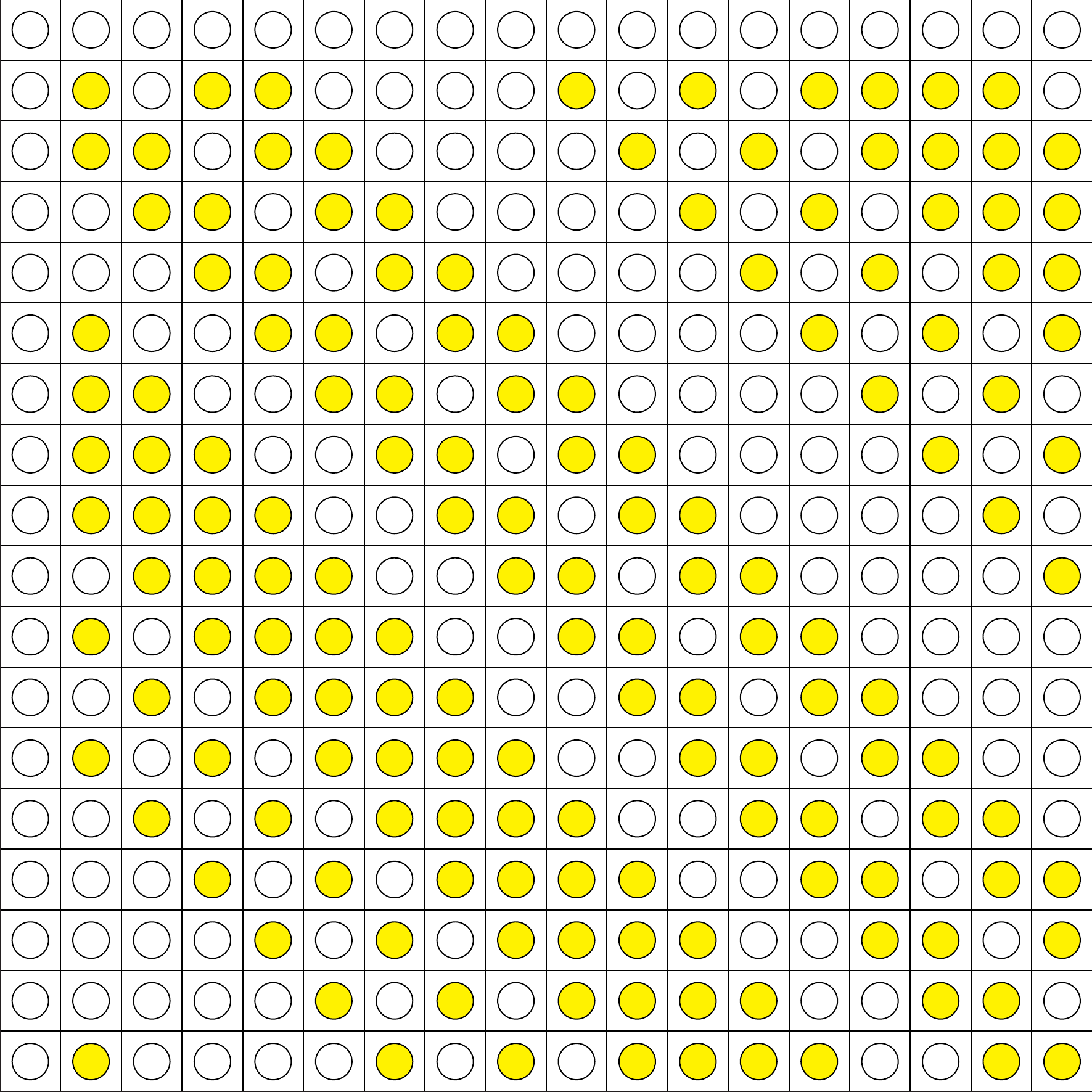}
\caption{Initial configuration of lights generating $R_{18} \geq122$}
\end{figure}%
%

\begin{figure}[H]
\centering\includegraphics[width=5.9cm]{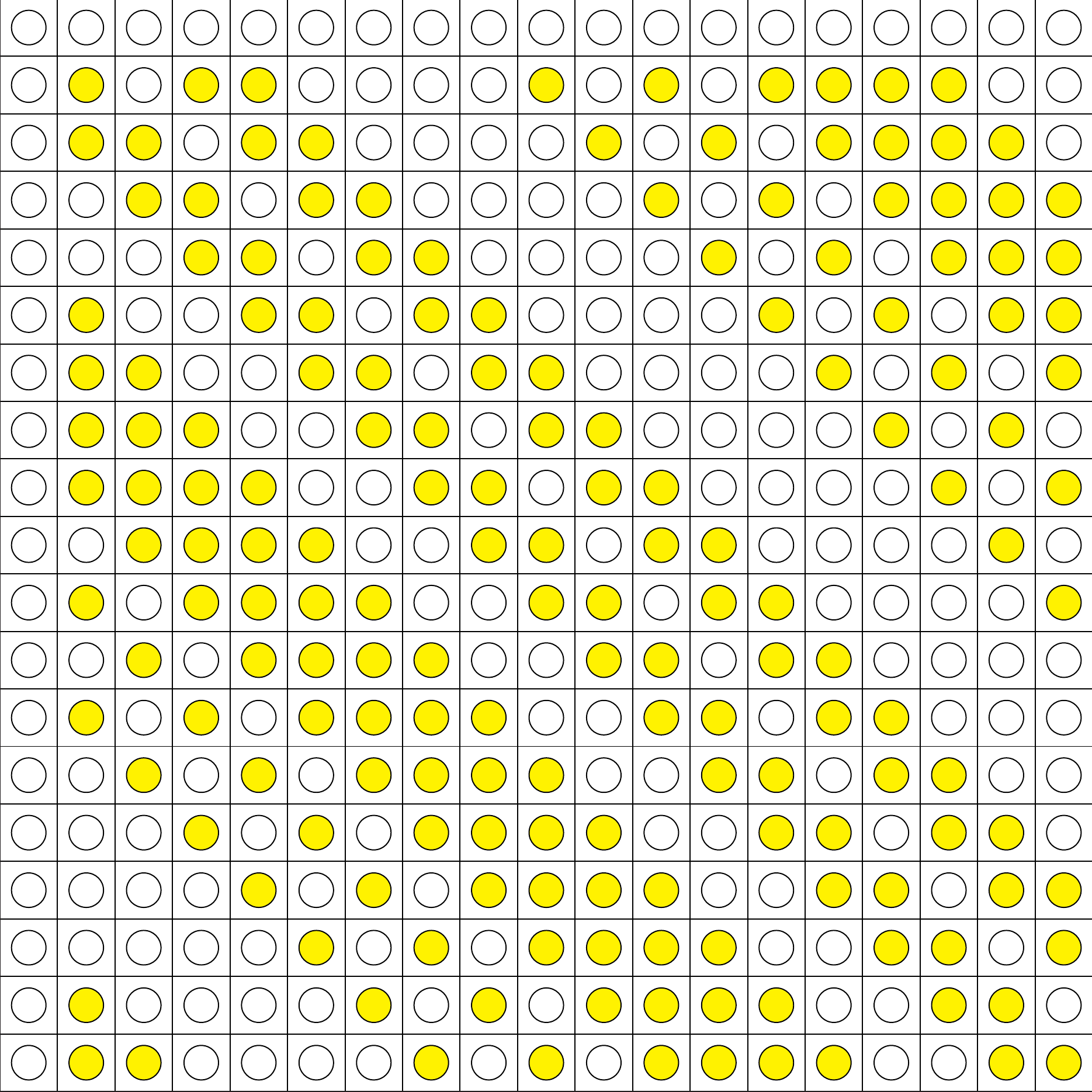}
\caption{Initial configuration of lights generating $R_{19} \geq139$}
\end{figure}%

\end{document}